\documentclass[12pt]{amsart}

\usepackage{enumerate}
\usepackage{amssymb}
\usepackage{amsthm}
\usepackage{amsmath}
\usepackage{color}
\usepackage{amsmath}
\usepackage{amssymb}
\usepackage{amsfonts,mathrsfs}
\usepackage{lipsum}
\usepackage{array}

\theoremstyle{plain}

\newcommand{\C}{\mathbb{C}}
\newcommand{\T}{\mathbb{T}}
\newcommand{\D}{\mathbb{D}}

\numberwithin{equation}{section}
\theoremstyle{plain}
\newtheorem{Proposition}[equation]{Proposition}
\newtheorem{Corollary}[equation]{Corollary}
\newtheorem*{Corollary*}{Corollary}
\newtheorem{Theorem}[equation]{Theorem}
\newtheorem*{Theorem*}{Theorem}
\newtheorem{Lemma}[equation]{Lemma}
\theoremstyle{definition}

\newtheorem{Example}[equation]{Example}
\newtheorem{Remark}[equation]{Remark}

\renewcommand{\leq}{\leqslant}
\renewcommand{\geq}{\geqslant}
\renewcommand{\subset}{\subseteq}
\renewcommand{\supset}{\supseteq}

\renewcommand{\vec}[1]{{\bf #1}}

\begin{document}

\title{Inner functions in reproducing kernel spaces}
\author[Cheng]{Raymond Cheng}
\address{Department of Mathematics and Statistics,
  Old Dominion University,
  Norfolk, VA 23529}
  \email{rcheng@odu.edu}

\author[Mashreghi]{Javad Mashreghi}
\address{D\'epartement de math\'ematiques et de statistique, Universit\'e laval, Qu\'ebec, QC, Canada, G1V 0A6}
\email{javad.mashreghi@mat.ulaval.ca}

\author[Ross]{William T. Ross}
	\address{Department of Mathematics and Computer Science, University of Richmond, Richmond, VA 23173, USA}
	\email{wross@richmond.edu}


\begin{abstract}
In BeurlingÕs approach to inner functions  for the shift operator $S$ on the Hardy space $H^2$, a function $f$ is inner when $f \perp S^n f$ for all $n \geq 1$.  Inspired by this approach, this paper develops a notion of an inner vector $\mathbf{x}$ for any operator $T$ on a Hilbert space, via the analogous condition $\mathbf{x} \perp T^n \mathbf{x}$ for all $n \geq 1$. We study these inner vectors in a variety of settings. Using Birkhoff-James orthogonality, we extend this notion of inner vector for an operator on a Banach space. We then apply this development of inner function to recast a theorem of Shapiro and Shields to discuss the zero sets for functions in Hilbert spaces, as well as obtain a corresponding result for zero sets for a wide class of Banach spaces. 
\end{abstract}

\dedicatory{In memory of S.\ Shimorin.}

\subjclass[2010]{Primary: , Secondary: }


\thanks{This work was supported by NSERC (Canada).}
\maketitle


\section{Introduction} 
Inspired by Beurling's analysis  of the structure of the shift invariant subspaces of the classical Hardy space $H^2$ \cite{Beurling, Dur}, and by similar analysis in other settings \cite{MR1440934, Dima, MR936999, MR1259923, MR1145733, MR0145082},  we explored a notion of ``inner function'' in the sequence space $\ell^{p}_{A}$ and used it to characterize its zero sets \cite{ zeroslp, MR3686895}.  As this ``Beurling approach'' seems to be ubiquitous, we will survey a method from \cite{MR0145082} to the setting of reproducing kernel Hilbert spaces of analytic functions, as we head towards an analogous result for Banach spaces of analytic functions.

Broadly speaking, we start with a Banach space $\mathscr{X}$ of analytic functions on a bounded planar domain $\Omega$ for which, among some mild technical conditions (see Section \ref{s3}),  the shift operator $(S_{\mathscr{X}}  f)(z)  = z f(z)$ is well defined and continuous. We will examine a notion of ``orthogonality'' $f \perp_{\mathscr{X}} g$ for $f, g \in \mathscr{X}$ due to Birkhoff and James \cite{Jam} (see Section \ref{section7}) and use this orthogonality to define an {\em $S_{\mathscr{X}}$-inner function} to be an  $f \in \mathscr{X}\setminus\{0\}$ for which 
$$f \perp_{\mathscr{X}} S_{\mathscr{X}}^n f, \quad n \geq 1.$$ When $\Omega$ is the open unit disk $\D$ and $\mathscr{X}$ is the classical Hardy space $H^2$, basic Fourier analysis will show that an $S_{H^2}$-inner function is a bounded analytic function on $\D$ for which the radial boundary function has constant modulus almost everywhere, in agreement with the classical and well-known notion of inner. Similarly defined inner functions were explored in other spaces \cite{MR1440934, Dima, zeroslp, Seco}.  As a topic to be explored in future work, a more general notion of $T$-inner vector will be presented in this paper, in which $T$ is a bounded linear transformation on a Banach space $\mathscr{X}$, and a vector $\vec{x} \in \mathscr{X}$ is said to be $T$-inner if $\vec{x} \perp_{\mathscr{X}} T^n \vec{x}$ for all $n \geq 1$.

This abstract notion of ``inner'' arises naturally in prediction theory for norm-stationary processes.  We say that a nonzero sequence $\{X_k\}_{k \in \mathbb{Z}}$ in a Banach space $\mathscr{X}$ is {\em norm stationary} when
\begin{equation}\label{brtr3}
\Big\|\sum_{j = 1}^{m} a_j X_{k_j}\Big\| = \Big\|\sum_{j = 1}^{m} a_j X_{k_j + t}\Big\|
\end{equation}
for all $t \in \mathbb{Z}$, coefficients $a_j \in \C$, and indices $k_j \in \mathbb{N}$.  The identity in \eqref{brtr3} induces an isometry $T$ on
$$\mathscr{M} := \bigvee\{X_0, X_1, X_2, \ldots\},$$
 the closed linear span of the sequence $\{X_k\}_{k \geq0}$, for which
\[
T X_k = X_{k+1}, \quad k \geq 0.
\]
Writing $\widehat{X}_0$ for a metric projection (nearest point) of $X_0$ onto $T\mathscr{M}$, one can show that the vector $X_0 - \widehat{X}_0$ is $T$-inner on $\mathscr{M}$.
This construction appears in studies involving norm-stationary processes with infinite variance
\cite{CR,CR2,MP}, extending, in part, the extensive literature on stationary Gaussian processes.  In particular, the results from \cite{MP} seek to find a Wold-like decomposition in this setting.

This paper is structured as follows. In Section \ref{s2} we discuss a general notion of a $T$-inner vector, where $T$ is a bounded linear transformation on a Hilbert space, and give a variety of examples, and encourage the reader to investigate further. In Section \ref{EP}  we develop some basic properties of $T$-inner vectors and show in Proposition \ref{66zczxc} that all $T$-inner vectors take a particular form.

In Section \ref{s3} we apply this notion of $T$-inner to recast some work of Shapiro and Shields \cite{MR0145082} (in which the concept of inner also has its roots in the work of Beurling),  in terms $S_{\mathscr{H}}$-inner functions, to characterize the zero sets of a Hilbert space of analytic functions on a bounded planar domain (see Theorem \ref{v98u24tprefewq23re}). This will lead us in several directions. First, we explore whether the $S_{\mathscr{H}}$-inner function associated with a polynomial has extra zeros.  Indeed, with the Hardy space $H^2$, the inner factor of a function in $H^2$ has exactly all of the zeros of the original function, and no others. In Section \ref{mzkqwwowowo} we develop conditions (see Theorem \ref{extra1}) for which the $S_{\mathscr{H}}$-inner function $J$ associated with an $f \in \mathscr{H}$ (where $\mathscr{H}$ is a Hilbert space of analytic functions on a bounded planar domain) has only the zeros of $f$, and no others. In particular, our result applies to the shift operator on the well-known Dirichlet space (see Corollary \ref{77w9e8r7ew98}) as well as shift operator on a space studied by Korenblum (see Corollary \ref{xhbmdfer4terweqwa}). 

Second, we investigate the connection between inner functions and zero sets.  In particular, we encounter the phenomenon of an $S_{\mathscr{H}}$-inner function $J$ having ``extra zeros,'' that is, zeros in addition to a prescribed set.  The existence of such extra zeros was first demonstrated in \cite{MR1284108}, where $\mathscr{H}$ was a weighted Bergman space. In Section \ref{ExtraZ} we give a large class of spaces $\mathscr{H}$ for which the $S_{\mathscr{H}}$-inner function associated with a linear polynomial has extra zeros. 

Third, so far, we have focused on Hilbert spaces. In our final two sections we develop, via Birkhoff-James orthogonality, notions of ``inner'' for operators on Banach spaces.  Our concept of inner will coincide with the classical definition for the Hardy classes $H^p$, when $p \in (1, \infty)$.  In addition, we discuss the zero sets for Banach spaces of analytic functions on a planar domain, and prove an extension of the Shapiro-Shields result.


\section{Inner vectors in Hilbert spaces}\label{s2}

Let us begin with a discussion of $T$-inner vectors for Hilbert space operators $T$, where one can take a very broad approach. We will see later in the Banach space setting that some restrictions become necessary in order for the definitions to make sense. 

Let $\mathscr{H}$ be a complex Hilbert space with inner product $\langle\cdot,\cdot\rangle$, and let $T$ be a bounded linear operator on $\mathscr{H}$. We say a vector $\vec{v} \in \mathscr{H} \setminus \{0\}$ is {\em $T$-inner} when
$$\vec{v} \perp T^n \vec{v}, \quad n \geq 1.$$ For a vector $\vec{w} \in \mathscr{H}$, let 
\begin{equation}\label{3487ytrgeifodjsvca}
[\vec{w}]_{T} := \bigvee \{\vec{w}, T \vec{w}, T^2 \vec{w}, \ldots\}
\end{equation}
denote the $T$-invariant subspaces generated by $\vec{w}$. When the context is clear we will use $[\vec{w}]$ in place of  $[\vec{w}]_{T}$. Observe that $\vec{v}$ is $T$-inner precisely when 
$\vec{v} \perp [T \vec{v}]_{T}$. Here are a few examples of $T$-inner vectors. 

\begin{Example}\label{nnnnnc}
Suppose that $T$ is the shift operator $(T f)(z) = z f(z)$ on the classical Hardy space $H^2$ \cite{Dur}. Via standard theory of radial boundary values, the inner product on $H^2$ can be written as the integral 
\begin{equation}\label{lkjlkjdlkfj}
\langle f, g \rangle = \int_{0}^{2 \pi} f(e^{i \theta}) \overline{g(e^{i \theta})} \frac{d \theta}{2 \pi}, \quad f, g \in H^2.
\end{equation}
Thus a function (vector) $f \in H^2 \setminus \{0\}$ is $T$-inner precisely when
$$0 = \langle f, T^n f\rangle = \int_{0}^{2 \pi} |f(e^{i \theta})|^2 e^{-in \theta} \frac{d \theta}{2 \pi}, \quad n \geq 1.$$
The equation above, along with its complex conjugate, show $f$ is $T$-inner precisely when all but the zeroth Fourier coefficients of $|f|^2$ vanish. This implies that the function $\theta \mapsto |f(e^{i \theta})|$ is constant almost everywhere. 
The condition ``$|f|$ has constant radial limit values almost everywhere on the unit circle'', is the classical definition of inner \cite{Dur} -- though one usually normalizes things so that inner means $|f(e^{i \theta})| = 1$ for almost every $\theta$. We will refer to this notion of inner as {\em classical inner}.
\end{Example}

\begin{Example}\label{nnnnxx}
Suppose that $(T f)(z) = z^2 f(z)$, the {\em square} of the unilateral shift on $H^2$. Then, with a similar analysis as in the previous example,  $f \in H^2$ is $T$-inner when 
$$\int_{0}^{2 \pi} |f(e^{i \theta})|^2 e^{2 i k \theta} \frac{d \theta}{2 \pi} = 0, \quad k \in \mathbb{Z} \setminus \{0\},$$ 
though it is somewhat unclear what to glean from this condition. Certainly any classical inner function is a $T$-inner function. 
However, functions like $f(z) = a + b z$, which are not classical inner when $a$ and $b$ are both nonzero, is a $T$-inner function. Observe that this class of $T$-inner functions is closed under multiplication by classical inner functions.

With a little extra effort, and transferring the problem to a different venue, we can describe the $T$-inner functions more explicitly. 
Indeed, if
$$H^2 \oplus H^2 := \{f \oplus g: f, g \in H^2\}$$ with norm 
$$\|f \oplus g\|_{H^2 \oplus H^2}^{2} := \int_{0}^{2 \pi} |f(e^{i \theta})|^2 \frac{d \theta}{2 \pi} + \int_{0}^{2 \pi} |g(e^{i \theta})|^2 \frac{d \theta}{2 \pi},$$
then  the operator 
$$U: H^2 \to H^2 \oplus H^2,$$ defined by 
\begin{equation}\label{98rguioew}
U\Big(\sum_{n = 0}^{\infty} a_n z^n\Big) = \Big(\sum_{n = 0}^{\infty} a_{2 n} z^n, \sum_{n = 0}^{\infty} a_{2 n + 1} z^n\Big)
\end{equation} is unitary.  
 Furthermore, if $(S f)(z) = z f(z)$ is the shift on $H^2$, we have 
$$S \oplus S: H^2 \oplus H^2 \to H^2 \oplus H^2, \quad (S \oplus S)(f \oplus g) = (S f) \oplus (S g),$$
and one can show that $U S^2 = (S \oplus S) U$.
Thus $f \in H^2$ is $S^2$-inner, if and only if $U f \in H^2 \oplus H^2$ is $S \oplus S$-inner. If 
$U f = f_1 \oplus f_2$ as in \eqref{98rguioew}, then $f$ is $S^2$-inner when
\begin{align*}
0 & = \langle (f_1 \oplus f_2, (S \oplus S)^n (f_1 \oplus f_2)\rangle_{H^2 \oplus H^2}\\
& = \langle f_1 \oplus f_2, (S^n f_1) \oplus (S^n f_2)\rangle_{H^2 \oplus H^2}\\
& = \int_{0}^{2 \pi} |f_1(e^{i \theta})|^2 e^{-i n \theta} \frac{d \theta}{2 \pi} + \int_{0}^{2 \pi} |f_2(e^{i \theta})|^2 e^{-i n \theta} \frac{d \theta}{2 \pi}.
\end{align*}
The above equation, along with its complex conjugate,  shows that $|f_1|^2 + |f_2|^2$ is (almost everywhere) constant on the circle.  We leave it to the reader to show that $U^{-1} (f_1 \oplus f_2)$ is equal to $f_1(z^2) + z f_2(z^2)$ and thus $f \in H^2$ is $S^{2}$-inner if and only if 
$$f(z) = f_1(z^2) + z f_{2}(z^2),$$ where $f_1, f_2 \in H^2$ with $|f_1|^2 + |f_2|^2$ is constant almost everywhere on $\T$. 

This example only scratches the surface of a much wider (and deeper) theory of shifts of higher multiplicity and the well-developed Beurling-Lax theorem \cite{Hoffman}.
\end{Example}

\begin{Example}\label{oiuweoiruew}

The previous example can be extended even further to $T = T_{\phi}$, $\phi \in H^{\infty}$, is an analytic Toeplitz operator on $H^2$ with symbol $\phi$, i.e., $T_{\phi} f = \phi f$. Here $f \in H^2 \setminus \{0\}$ is $T_{\phi}$-inner when 
$$\int_{0}^{2 \pi} |f(e^{i \theta})|^2 \overline{\phi(e^{i \theta})}^{n} \frac{d \theta}{2 \pi} = 0, \quad n \geq 1.$$
Of course, when $\phi(0) = 0$, then any (classical) inner function is $T_{\phi}$ inner,
and this class is also closed under multiplication by classical inner functions. In general, what are the $T_{\phi}$-inner functions? 

Let us work out a particular example. Suppose that $\phi$ is a Riemann map from $\mathbb{D}$ onto a simply connected domain $G$ with smooth boundary $\Gamma$. Then, with $ds_{\mathbb{T}}$ denoting arc length measure on $\mathbb{T}$, $ds_{\Gamma}$ denoting arc length measure on $\Gamma$, and $\psi = \phi^{-1}$, we see, via a change of variables,  that a unit vector $f \in H^2$ is $T_{\phi}$-inner when 
\begin{align*}
0 & = \int_{\mathbb{T}} |f(\zeta)|^2 \overline{\phi(\zeta)}^{n} ds_{\mathbb{T}}(\zeta)\\
& = \int_{\Gamma} |f(\psi(w)|^2 |\psi'(w)| \overline{w}^{n} ds_{\Gamma}(w), \quad n \geq 1.
\end{align*}
Using the (harmless) assumption that $f$ is a unit vector, we see that 
$$\int_{\Gamma} (|f(\psi(w)|^2 |\psi'(w)| - 1) \overline{w}^{n} ds_{\Gamma}(w) = 0, \quad n \geq 0.$$ Taking the complex conjugate of the above expression we see the measure 
$$(|f \circ \psi|^2 |\psi'|  - 1)ds_{\Gamma}$$ 
annihilates $w^{n}$ and $\overline{w}^{n}$ for all $n \geq 0$. Standard harmonic analysis will show that this measure must be the zero measure and so 
$$|f \circ \psi|^2 |\psi'|  = 1$$ almost everywhere on $\Gamma$. Consequently, we see that 
$$|f|^2 |\psi' \circ \phi| = 1$$ almost everywhere on $\mathbb{T}$. But since 
$$\psi'\circ \phi = \frac{1}{\phi'}$$ we see that $f/\sqrt{\phi'}$ is a classical inner function. In summary, $f$ is $T_{\phi}$-inner if and only if $f/\sqrt{\phi'}$ is a classical inner function.  We thank Dima Khavinson for pointing this out to us.

For a particularly simple example, consider the case where 
$$\phi(z) = \frac{z - w}{1 - \overline{w} z}, \quad w \in \mathbb{D}.$$
Here $\phi$ is a simple Blaschke factor (which is an automorphism of $\mathbb{D}$). Since 
$$\phi'(z) = \frac{1 - |w|^2}{(1 - \overline{w} z)^2},$$ the $T_{\phi}$ inner functions in this case take the form 
$$C \frac{j(z)}{1 - \overline{w} z},$$ where $C \in \mathbb{C}$ and $j$ is is a classical inner function.

%

\end{Example}

\begin{Example}
If $(T f)(x) = x f(x)$  on $L^2[0, 1]$, it is an easy exercise to show that there are no (non-zero) $T$-inner vectors. Indeed, if 
$$\langle f, x^n f\rangle = \int_{0}^{1} x^n |f(x)|^2 dx = 0, \quad n \geq 1,$$
then all the polynomials annihilate the measure $x |f(x)|^2 dx$ and an argument using the Weierstrass approximation and the Riesz representation theorems will show that $f = 0$ (almost everywhere). 
\end{Example}

\begin{Example} 
Let 
$$(T f)(x) = \int_{0}^{x} f(t) dt,$$
be the {\it Volterra operator} on $L^2[0, 1]$.  Let us establish that there are no non-zero $T$-inner vectors. 
 By a well-known result  \cite{MR0358396}, every invariant subspace of the Volterra operator takes the form $\chi_{[a, 1]} L^2[0, 1]$ for some $a \in [0, 1]$. Thus 
 $$[T f] = \chi_{[a, 1]} L^2[0, 1]$$ for some $a \in [0, 1]$. By the Lebesgue differentiation theorem, 
$f = \frac{d}{d x} T f$ almost everywhere and so $f \in \chi_{[a, 1]} L^2[0, 1]$. In other words, 
$f \in [T f]$, and since $f$ is $T$-inner, we have $f \perp f$.
This forces $f=0$,
and so there are no $T$-inner functions.
\end{Example}

\begin{Example}\label{ooivvcdvw1} 
Let $T$ denote the {\it compressed shift} $T f = P_{\Theta} (z f)$ on the model space $(\Theta H^2)^{\perp}$, where $\Theta$ is a classical inner function as in Example \ref{nnnnnc}. These compressed shifts have been well studied and  serve as models for certain types of contractions on Hilbert spaces \cite[Ch,~9]{MR3526203}. Here an $f \in (\Theta H^2)^{\perp}$ is $T$-inner when 
\begin{align*}
0 &= \langle f, T^n f\rangle\\
& = \langle f, P_{\Theta}(z^n f)\rangle\\
& = \langle P_{\Theta} f, z^n f\rangle\\
& = \langle f, z^n f\rangle\\
& =  \int_{0}^{2 \pi} |f(e^{i \theta})|^2 e^{-in \theta} \frac{d \theta}{2 \pi}, \quad n \geq 1.
\end{align*}
As in Example \ref{nnnnnc}, this says that $f$ must have constant modulus on the unit circle and thus be a classical inner function. However, $f$  must also belong to the model space $(\Theta H^2)^{\perp}$. This extra condition places a restriction on $\Theta$, namely $\Theta(0) = 0$, and on $f$, namely $f$ must be an inner divisor of $\Theta/z$ \cite[p.~177]{MR3526203}. 
\end{Example}

\begin{Example}
Continuing with Example \ref{ooivvcdvw1}, one can consider the special case where $\Theta(z) = z^n,$ $n\geq 1$. Here the model space takes the form 
$$(z^n H^2)^{\perp} = \bigvee\{1, z, z^2, \ldots, z^{n - 1}\}$$ and the matrix representation of the compressed shift $T f = P_{\Theta} (z f)$ with respect to the orthonormal basis $\{1, z, z^2, \ldots, z^{n - 1}\}$ for $(z^n H^2)^{\perp}$ becomes 
$$ \left(
\begin{array}{ccccc}
 0 &  &  &  & \\
 1 & 0 &  &  & \\
  & 1 & \ddots &  & \\
  &  & \ddots & 0 & \\
   &  &  & 1 & 0\\
\end{array}
\right)$$
(see \cite{MR3526203}).
The powers of the above matrix just move the $1$s on the sub-diagonal to the succeeding sub-diagonals (until the matrix becomes the zero matrix) and from here one can see that the $T$-inner vectors are $\vec{v} = c\, \vec{e}_j$, for $j = 0, 1, \ldots, n - 2$, where $\vec{e}_{j}$ is the standard basis vector. Notice how this corresponds to the $T$-inner vectors 

\[
     f(z) = cz^k,\quad k=0, 1,\ldots, n-1.
\]
from the previous example (the inner divisors of $z^{n - 1}$).
\end{Example}

\begin{Example}  
In the previous example if $\Theta(z) = z^4$, then the model space becomes $(z^4 H^2)^{\perp} = \bigvee\{1, z, z^2, z^3\}$ and the matrix of the compressed shift is
$$\left(
\begin{array}{cccc}
 0 & 0 & 0 & 0 \\
 1 & 0 & 0 & 0 \\
 0 & 1 & 0 & 0 \\
 0 & 0 & 1 & 0 \\
\end{array}
\right).$$
If $T$ is the {\em square} of the compressed shift, then $T$ has matrix representation 
$$\left(
\begin{array}{cccc}
 0 & 0 & 0 & 0 \\
 0 & 0 & 0 & 0 \\
 1 & 0 & 0 & 0 \\
 0 & 1 & 0 & 0 \\
\end{array}
\right).$$
If $\vec{v} = (z_1, z_2, z_3, z_4) \in \C^4$, one can quickly check that $\vec{v}$ is $T$-inner if and only if 
$$z_3 \overline{z_1} + z_4 \overline{z_2} = 0.$$ In terms of a function in the model space, this says, for example, that $f(z) = a + b z^3$ is $T$-inner for any $a, b \in \C$. 
\end{Example}

\begin{Example}\label{00uuucbbjbbjb}
Let $(T f)(z) = z f(z)$ be the unilateral shift on the Dirichlet space $\mathscr{D}$ of analytic functions 
$f(z) = \sum_{n \geq 0} a_n z^n$ on $\D$ for which 
\begin{equation}\label{nncowef}
\sum_{n \geq 0} (1 + n) |a_n|^2 < \infty.
\end{equation} The above quantity defines the square of the norm on $\mathscr{D}$.
In \cite{Seco, MR0145082} they discussed the $T$-inner functions. The reproducing kernel for $\mathscr{D}$ is $$k_{w}(z) = \frac{1}{\overline{w} z} \log \Big(\frac{1}{1  - \overline{w} z}\Big), \quad w, z \in \D,$$ and the function $$f(z) = k_{w}(w) - k_{w}(z)$$ is $T$-inner. 
\end{Example}

\begin{Example}\label{popoweripeo}
Let $(T f)(z) = z f(z)$ be the unilateral shift on the Bergman space $\mathscr{B}$ of analytic functions 
$f(z) = \sum_{n \geq 0} a_n z^n$ on $\D$ for which 
$$\sum_{n \geq 0} \frac{|a_n|^2}{n + 1} < \infty.$$
The above quantity defines the square of the norm on $\mathscr{B}$. The $T$-inner functions were discussed in \cite{MR1440934}.
 As in the Dirichlet space example, if 
$$k_{w}(z) = \frac{1}{(1 - \overline{w} z)^2}$$ denotes the reproducing kernel for $\mathscr{B}$ then 
$k_{w}(w) - k_{w}(z)$ is a $T$-inner function.
\end{Example}

\begin{Example}\label{erewrwe}
Consider the space $H^{2}_{1}$ of analytic functions $f \in H^2$ whose first derivative $f'$ also belongs to $H^2$. This space, along with other associated spaces, was studied by Korenblum in \cite{MR0317073} in his work on ideals of algebras of analytic functions. The quantity 
$$|f(0)|^2 + \sum_{n \geq 1} n^2 |a_n|^2$$ defines the square of the norm on this space. 
 This is a reproducing kernel Hilbert space with kernel 
$$k_{w}(z) = 1 + \sum_{n \geq 1} \frac{\overline{w}^n z^n}{n^2}.$$ The shift operator $(T f)(z)= z f(z)$ turns out to be continuous on $H^{2}_{1}$ and, as with previous two examples, $k_{w}(w) - k_{w}(z)$ is a $T$-inner function. 
\end{Example}

\begin{Example}
We point out that $T$-inner functions for $(Tf)(z) = zf(z)$ in other weighted Hardy spaces were studied in \cite{Dima}.
\end{Example}

Observe that in the four previous examples of the shift on the Dirichlet space, the Bergman space, $H^{2}_{1}$, and other weighted spaces, the respective $T$-inner functions look quite different.


\section{Elementary Properties}\label{EP}

Here are some routine but nevertheless interesting facts about $T$-inner vectors. Recall the definition of $[T \vec{v}]$ from \eqref{3487ytrgeifodjsvca}.

\begin{Proposition}\label{66zczxc}
Suppose that $T$ is a bounded linear transformation on a Hilbert space $\mathscr{H}$ and $\vec{v}$ is any vector in $\mathscr{H}$.  Let $P_{[T \vec{v}]}$ be the orthogonal projection onto the subspace
$[T \vec{v}]$.
Then the vector $\vec{v} - P_{[T \vec{v}]} \vec{v}$ is $T$-inner (or zero), and every $T$-inner vector arises  in this way. 
\end{Proposition}

\begin{proof}
Observe that for any two vectors $\vec{u}, \vec{v}$ in a Hilbert space $\mathscr{H}$  
we have 
\begin{equation}\label{po09oifrgefd}
\vec{u} \perp \vec{v} \iff \|\vec{u} + \alpha \vec{v}\| \geq \|\vec{u}\|, \quad \alpha \in \C.
\end{equation}
To see this, use the Pythagorean theorem for one direction and the definition of the orthogonal projection of $\vec{u}$ onto $\vec{v}$ for the other. 

By the definition of the orthogonal projection $P_{[T \vec{v}]}$, we know that 
$$\vec{v} - P_{[T \vec{v}]} \vec{v} \perp [T \vec{v}]$$ and so for any $n \geq 1$ we can use \eqref{po09oifrgefd} to see that 
$$\|(\vec{v} - P_{[T \vec{v}]} \vec{v})  - \alpha T^{n} (\vec{v} - P_{[T \vec{v}]} \vec{v})\| \geq \|\vec{v} - P_{[T \vec{v}]} \vec{v}\|, \quad \alpha \in \C.$$
Another application of \eqref{po09oifrgefd} yields
$$\vec{v} - P_{[T \vec{v}]} \vec{v} \perp T^{n} (\vec{v} - P_{[T \vec{v}]} \vec{v})$$ which says that $\vec{v} - P_{[T \vec{v}]} \vec{v}$ is $T$-inner. 

Now suppose that $\vec{v}$ is $T$-inner. By the definition of $T$-inner, $\vec{v} \perp \vec{z}$ for all $\vec{z} \in [T \vec{v}]$ which implies 
$$\|\vec{v}\| \leq \|\vec{v} - \vec{z}\|, \quad \vec{z} \in [T \vec{v}].$$
By the uniqueness of $P_{[T \vec{v}]} \vec{v}$ as a vector satisfying the above inequality, we see that $P_{[T \vec{v}]} \vec{v} = \vec{0}$ and so the $T$-inner vector $\vec{v}$ has the desired form $\vec{v} = \vec{v} - P_{[T \vec{v}]} \vec{v}$.
\end{proof}

\begin{Remark}
This proposition suggests a possible avenue to describe the $T$-inner vectors. Indeed, if $\{\vec{u}_1, \vec{u}_2, \ldots\}$ is an orthonormal basis for $[T \vec{v}]$, then Proposition \ref{66zczxc} says that every $T$-inner function can be described as 
\begin{equation}\label{03495fsd}
\vec{v} - \sum_{j \geq 1} \langle \vec{v}, \vec{u}_j\rangle \vec{u}_j.
\end{equation}
 Though this approach might seem initially appealing, this is not always a tractable problem. 
For example, when  $T = T_{\phi}$, $\phi \in H^{\infty}$, is an analytic Toeplitz operator on $H^2$, as in Example \ref{oiuweoiruew}, the above analysis requires a description of 
$$[T_{\phi} f] = \bigvee\{\phi f, \phi^2 f, \phi^3 f, \ldots\}$$ which can be extremely complicated.  

When $\phi(z) = z$, things become much easier in that Beurling's theorem \cite{Dur} says that $[z f] = z I_{f} H^2$, where $I_f$ is the (classical) inner factor of $f$. Moreover, due to the fact that each of the functions $z^{n + 1} I_{f}$ has unimodular boundary values, along with Beurling's theorem, the set $\{z^{n + 1} I_f: n \geq 0\}$ is an orthonormal basis for $z I_f H^2$. Furthermore, following the formula in \eqref{03495fsd}, we have 
$$\langle f, z^{n + 1} I_{f}\rangle = \widehat{O_{f}}(n + 1),$$
where $\widehat{O_{f}}(n + 1)$ is $(n + 1)$st Fourier coefficient of the outer factor $O_{f}$ of $f$. Thus we obtain the curious fact that
\begin{equation}\label{63443yurfv}
f - \sum_{n = 0}^{\infty} \widehat{O_{f}}(n + 1) z^{n + 1} I_{f} = \widehat{O}_f(0)I_f
\end{equation}
is inner (in the classical sense) for any nonzero $f \in H^2$ and moreover,  any inner function arises in this fashion. Note that when $f$ is inner then $\widehat{O_{f}}(n + 1) = 0$ for all $n \geq 0$ and so the expression in \eqref{63443yurfv} simply reduces to $f$. When $f$ is outer, then $I_{f} = 1$ and \eqref{63443yurfv} becomes the constant function $\widehat{O_{f}}(0)$ which, according to our definitions, is inner. 
\end{Remark}

\begin{Proposition}
A vector $\vec{v} \in \mathscr{H}$ is $T$-inner if and only if $\vec{v}$ is $T^{*}$-inner. 
\end{Proposition}

\begin{proof}
For any $n \geq 1$ we have 
$$\langle \vec{v}, T^{n} \vec{v}\rangle = \langle {T^{*}}^n \vec{v}, \vec{v}\rangle.$$ This shows that $\vec{v}$ is $T$-inner if and only if $\vec{v}$ is $T^{*}$-inner. 
\end{proof}

Though the proposition above seems to be a triviality, we mention it since in the Banach space setting the $T$-inner vectors and the $T^{*}$-inner vectors are from different spaces (see Proposition \ref{nnnnuuhuhuh}).


\section{Application: Zero sets for reproducing kernel Hilbert spaces}\label{s3}

In exploring the zero sets of functions in the Dirichlet space $\mathscr{D}$ (recall the definition from \eqref{nncowef}), Shapiro and Shields \cite{MR0145082} constructed solutions to certain extremal problems. As a consequence of their investigations, they developed necessary and sufficient conditions on a sequence of points in $\mathbb{D}$ to be the set of zeros of a non-trivial function from $\mathscr{D}$. 
(Towards a Banach space generalization of this, see Section \ref{section7}.)
We now recast the Shapiro-Shields construction in the language of $S$-inner functions on a more general class of Hilbert spaces of analytic functions and obtain a characterization of zero sets.   We will also begin to examine when these $S$-inner functions have extra zeros. 

Suppose $\Omega$ is a bounded domain in $\C$ with $0 \in \Omega$. Also suppose that $\mathscr{H}$ is a Hilbert space of (scalar-valued) analytic functions on $\Omega$ satisfying the following properties:

\noindent For every nonnegative integer $j$, and every $w\in\Omega$, there exists a constant $C = C(j, w)$ such that 
\begin{equation}\label{P1}
   |f^{(j)}(w)| \leq C\|f\|, \quad f \in \mathscr{H};
\end{equation}
\begin{equation}\label{P2}
f \in \mathscr{H} \implies z f(z) \in \mathscr{H};
\end{equation}
\begin{equation}\label{P3}
\bigvee\{z^j: j \geq 0\} =  \mathscr{H};
\end{equation}
\begin{equation}\label{P4}
w \in \Omega, f \in \mathscr{H} \implies (Q_{w}f)(z):= \frac{f(z) - f(w)}{z - w} \in \mathscr{H}
\end{equation}

The first property \eqref{P1} says that for each $w \in \Omega$, the point evaluation at $w$ of the $j$th order derivative of $f$ is continuous and so, by the Riesz representation theorem for Hilbert spaces, there is a $k_{j,w} \in \mathscr{H}$ (called a {\em reproducing kernel \cite{MR3526117}} for $\mathscr{H}$) for which  
$$f^{(j)}(w) = \langle f, k_{j, w}\rangle, \quad f \in \mathscr{H}.$$
When $j=0$ we write $k_{w}$ in place of $k_{0,w}$.

The closed graph theorem, together with the second property \eqref{P2}, shows that the shift operator 
$$S_{\mathscr{H}}: \mathscr{H} \to \mathscr{H}, \quad (S_{\mathscr{H}} f)(z) = z f(z),$$
is well defined and continuous on $\mathscr{H}$.  We included the hypothesis that $\Omega$ was a bounded domain from the beginning. However, the continuity of $S_{\mathscr{H}}$ along with the existence of reproducing kernels $k_{w}, w \in \Omega$, automatically gives us that $\Omega$ is a bounded domain. Indeed, it is a straightforward computation to show that 
$$S_{\mathscr{H}}^{*} k_{w} = \overline{w} k_{w}, \quad w \in \Omega.$$   It follows that $\{\overline{w}: w \in \Omega\}$ must belong to the spectrum of $S_{\mathscr{H}}^{*}$, which, by basic functional analysis, is a bounded set. Thus, at the end of the day, $\Omega$ is a bounded domain anyway.

Furthermore, the list of hypotheses \eqref{P1} -- \eqref{P4} is actually redundant in that we can deduce the first condition from the other three.  To see this, let $f\in\mathscr{H}$, and $w\in\Omega$.  By \eqref{P3}, $\mathscr{H}$ contains the constant function $1$, and so
\begin{align*}
    |f(w)|  &=  \frac{\|f(w)\|}{\|1\|} \\
      &\leq \frac{\|f - f(w)\|+\|f\|}{\|1\|} \\
      & \leq \frac{\|(S_{\mathscr{H}} - w I) Q_{w} f\| + \|f\|}{\|1\|}\\
      &\leq \frac{\|S_{\mathscr{H}} - w I\| \|Q_w f\|+\|f\|}{\|1\|} \\
      &\leq \frac{\|S_{\mathscr{H}} - w I\| \|Q_w\|  + 1}{\|1\|} \|f\|.
\end{align*}
From the Taylor series of $f$ about $w$, we see that
\[
     (Q_w f)(z) = f'(w) + \frac{f''(w)}{2!}(z-w) + \cdots.
\]
This shows that $(Q_w f)(w) = f'(w)$.  By the boundedness of $Q_w$, and of point evaluation as shown above, it must be that point evaluation at a derivative is bounded.  This result extends to derivatives of all orders, and \eqref{P1} follows.

We point out that many of the known Hilbert spaces of analytic functions (Hardy, Bergman, Dirichlet, etc) discussed previously satisfy conditions \eqref{P1} - \eqref{P4}.

If $(w_j)_{j \geq 1}$ is a sequence of points in $\Omega$ (repetitions allowed), then we say, for fixed $g \in \mathscr{H}$, that 
$$Z(g) = (w_j)_{j \geq 1},$$ when $w_j$ has multiplicity $r_j \geq 1$,  
$$g(w_j) = g'(w_j) = \cdots = g^{(r_j - 1)}(w_j) = 0$$ 
and 
$$g^{(r_j)}(w_j) \not = 0$$
and 
$$g(w) \not = 0 \; \; \mbox{when} \; \; w \not \in (w_j)_{j \geq 1}.$$

We say that $(w_j)_{j \geq 1} \subset \Omega$ is a {\em zero set} for $\mathscr{H}$ if $Z(g) \supset (w_j)_{j \geq 1}$ for some $g \in \mathscr{H} \setminus \{0\}$.  Here, $g$ may have zeros in addition to the prescribed points $(w_j)_{j \geq 1}$.  
Obviously $(w_{j})_{j \geq 1}$ cannot be a zero set for $\Omega$ if it has an accumulation point in $\Omega$.

\begin{Lemma}\label{oirutoriug}
Suppose $p$ is a polynomial whose zeros 
$$W  = \{w_1, w_2, \ldots, w_n\},$$ repeated according to their multiplicity, belong to $\Omega$. Then 
$$
[p] := \bigvee \{S_{\mathscr{H}}^{j} p: j \geq 0\}
 = \{g \in \mathscr{H}: Z(g) \supset W\}.
$$
\end{Lemma}

\begin{proof}
By property \eqref{P1} we see that since $Z(p) = W$ then
\begin{equation}\label{vv188zlretM}
\bigvee \{S_{\mathscr{H}}^{j} p: j \geq 0\}
 \subset \{g \in \mathscr{H}: Z(g) \supset W\}.
 \end{equation} For the other inclusion, let $g \in \mathscr{H}$ with $Z(g) \supset W$. Observe that $n$ applications of property \eqref{P4} shows that $g/p \in \mathscr{H}$. Now use  condition \eqref{P3}, the density of the polynomials in $\mathscr{H}$, to produce a sequence of polynomials $q_n$ so that $q_n \to g/p$ in the norm of $\mathscr{H}$. Using the continuity of $S_{\mathscr{H}}$ (really the continuity of $p(S_{\mathscr{H}})$) we see that $p q_n \to g$ in $\mathscr{H}$. This yields $\supset$ in \eqref{vv188zlretM} which completes the proof. 
\end{proof}

Sticking to the same notation as before, taking into account the multiplicities of the $w \in W$,  we use the notation 
$$\bigvee \{k_{w}: w \in W\}$$ to include the linear span of $k_{w}$ along with 
$k_{s, w_j}$ for $0 \leq s \leq r_{w} - 1$. 

\begin{Lemma}\label{sldkfjsdfeq}
Suppose $p$ is a polynomial whose zeros 
$$W  = \{w_1, w_2, \ldots, w_n\},$$repeated according to their multiplicity, belong to $\Omega$. Then 
$$[p] = \left(\bigvee\{k_{w}: w \in W\}\right)^{\perp}.$$
\end{Lemma}

\begin{proof}
Suppose that 
$$g \in \left(\bigvee\{k_{w}: w \in W\}\right)^{\perp}.$$
The reproducing property of the kernels $k_{w}$ will show that $Z(g) \supset W$ and so Lemma \ref{oirutoriug} yields $g \in [p]$. Conversely, if $g \in [p]$ then $Z(g) \supset W$ and so $g$ has zeros with at least the correct multiplicities at the $w \in W$ and so 
$0 =  \langle g, k_{w}\rangle$. Thus $g \perp k_{w}$ for all $w \in W$ which proves the reverse inclusion. 
\end{proof}

We now recast a result of Shapiro and Shields \cite{MR0145082} to develop a criterion, based on $S_{\mathscr{H}}$-inner functions, for an infinite sequence $(w_j)_{j \geq 1} \subset \Omega \setminus \{0\}$ to be a zero set for $\mathscr{H}$. To this end, let 
$$W_n = \{w_1, w_2, \ldots, w_n\}$$ and 
$$f_n(z)  = \prod_{j = 1}^{n} \Big(1 - \frac{z}{w_j}\Big),$$
which belongs to $\mathscr{H}$ by \eqref{P3}. Define the function 
$$J_n = f_n - P_{[z f_n]} f_n,$$
where $P_{[z f_n]}$ is the orthogonal projection of $\mathscr{H}$ onto 
$$[z f_n] = \bigvee\{z^{j} f_n: j \geq 1\},$$ and note that Proposition \ref{66zczxc} shows that $J_n$ is $S_{\mathscr{H}}$-inner. For notational convenience we are using $[z f_n]$ in place of the more cumbersome $[S_{\mathscr{H}} f]$.

To compute $J_n$ somewhat explicitly, let  $$v_1, v_2, \ldots, v_n$$ denote the Gram-Schmidt normalization of the kernel functions
$$k_{w_1}, \ldots, k_{w_n},$$
where, as discussed earlier in this section, we include $k_{s, w}$ for $0 \leq s \leq r_{w} - 1$ if the multiplicity of $w$ is more than one.  
Note that 
$$\bigvee\{k_{w_j}: 1 \leq j \leq n\} = \bigvee\{v_j: 1 \leq j \leq n\}$$ and by Lemma \ref{sldkfjsdfeq},
$$(\bigvee\{v_j: 1 \leq j \leq n\})^{\perp} = \{f \in \mathscr{H}: Z(f) \supset W_n\}.$$

Now define 
$$v_0 = \frac{k_0 - \sum_{j = 1}^{n} \langle k_0, v_j\rangle v_j}{\|k_0 - \sum_{j = 1}^{n} \langle k_0, v_j\rangle v_j\|}.$$ Observe that $v_0 \not = 0$, since
 $$k_{0} \not \in \bigvee\{k_{w_j}: 1 \leq j \leq n\},$$
 and that 
$$v_0, v_1, \ldots, v_n$$ is an orthonormal basis basis for 
$$\bigvee\{k_0, k_{w_1}, \ldots, k_{w_n}\}.$$ By Lemmas \ref{oirutoriug} and \ref{sldkfjsdfeq}, 
\begin{align*}
(\bigvee\{v_j: 0 \leq j \leq n\})^{\perp} & = \{g \in \mathscr{H}: Z(g) \supset W_n \cup \{0\}\}\\
&  = [z f_n].
\end{align*}

Basic linear algebra shows that 
\begin{align*}
P_{[z f_n]} f_n & = f_n - \sum_{j = 0}^{n} \langle f_n, v_j\rangle v_j \\
& = f_n - \langle f_n, v_0\rangle v_0
\end{align*}
and thus 
\begin{align}
J_n & = f_n - P_{[z f_n]} \\
& = \langle f_n, v_0\rangle v_0 \nonumber\\
& = \Big\langle f_n, \frac{k_0 - \sum_{j = 1}^{n} \langle k_0, v_j\rangle v_j}{\|k_0 - \sum_{j = 1}^{n} \langle k_0, v_j\rangle v_j\|}\Big\rangle \frac{k_0 - \sum_{j = 1}^{n} \langle k_0, v_j\rangle v_j}{\|k_0 - \sum_{j = 1}^{n} \langle k_0, v_j\rangle v_j\|}\nonumber \\
& = \frac{k_0 - \sum_{j = 1}^{n} \langle k_0, v_j\rangle v_j}{\|k_0 - \sum_{j = 1}^{n} \langle k_0, v_j\rangle v_j\|^2}.\label{bbbbcbcb}
\end{align}
In the above calculation
note the use of the facts that $\langle f_n, v_j\rangle = 0$ for all $1 \leq j \leq n$ and $\langle f_n, k_0 \rangle = f_n(0) = 1$. This says that 
\begin{align}
\|J_n\|^2 & = \frac{1}{\|k_0 - \sum_{j = 1}^{n} \langle k_0, v_j\rangle v_j\|^2} \nonumber\\
& = \frac{1}{\|k_0\|^2 - \sum_{j = 1}^{n} |\langle k_0, v_j\rangle|^2}. \label{skdfhgsdkfld11}
\end{align} By Bessel's inequality, applied to  the denominator of the expression above, we have $\|J_n\| > 1/\|k_0\|$, and that $\|J_n\|$ is a non-decreasing sequence in $n$. 

Let $\Phi_n$ be the co-projection of $k_0$ onto 
$\{g \in \mathscr{H}: Z(g) \supset W_n\}$. Again, linear algebra will show that 
$$\Phi_n = \sum_{j = 1}^{n} \langle k_0, v_j \rangle v_j$$ and equations \eqref{bbbbcbcb} and \eqref{skdfhgsdkfld11} yield the identity
$$\Phi_n = k_0 - \frac{J_n}{\|J_n\|^2}.$$
By Bessel's inequality we have 
$$\|\Phi_n\|^2 = \sum_{j = 1}^{n} |\langle k_0, v_j\rangle|^2 \leq \|k_0\|^2.$$ 

We now present a technical lemma.

\begin{Lemma}
With the notation above,
$(w_j)_{j \geq 1}$ is a zero set for $\mathscr{H}$ if and only if 
$$\sup\{ \|\Phi_n\|: n \geq 1\} < \|k_0\|^2.$$ 
\end{Lemma} 

\begin{proof}
Let $W = (w_j)_{j \geq 1}$ and
$$\mathscr{H}(W)  := \{g \in \mathscr{H}: Z(g) \supset W\}.$$ 
From our previous discussions we now see that 
$$\bigvee \{v_j: j \geq 1\} = \bigvee \{k_{w_j}: j \geq 1\}$$ and 
$$(\bigvee \{k_{w_j}: j \geq 1\})^{\perp} = \mathscr{H}(W).$$ 
Also observe that 
$$\sup\{\|\Phi_n\|: n \geq 1\} = \sum_{j \geq 1} |\langle k_{0}, v_j\rangle|^2 = \|k_0\|^2$$ 
if and only if 
$$k_0 \in \bigvee \{k_{w_j}: j \geq 1\}$$
if and only if 
$$f \in \mathscr{H}(W) \implies f(0) = 0.$$

 Thus if $\mathscr{H}(W) \not = \{0\}$ then for some $n \geq 0$, $f(z)/z^n$ belongs to $\mathscr{H}(W)$ (note the use of property \eqref{P4}) and does not vanish at the origin. The result now follows. 
\end{proof}

Finally we note that $1 = J_{n}(0) = \langle J_n, k_0\rangle$ and so 
\begin{align*}
\|\Phi_n\|^2 & = \langle \Phi_n, \Phi_n\rangle\\
& = \Big\langle k_0 - \frac{J_n}{\|J_n\|^2}, k_0 - \frac{J_n}{\|J_n\|^2}\Big\rangle\\
& = \|k_0\|^2 - \frac{1}{\|J_n\|^2}.
\end{align*}
Putting this all together,  we obtain the identity 
$$(\|k_0\|^2-\|\Phi_n\|^2 ) \|J_n\|^2 = 1,$$ which means that 
$(w_j)_{j \geq 1}$ is a zero set for $\mathscr{H}$ if and only if 
$$\sup\{ \|J_n\| : n \geq 1\}< \infty.$$  This leads to the following result of Shapiro and Shields \cite{MR0145082}, expressed in terms of $S_{\mathscr{H}}$-inner functions, and extended to a wide class of reproducing kernel Hilbert spaces of analytic functions. 

\begin{Theorem}\label{v98u24tprefewq23re}
Let $(w_j)_{j \geq 1} \subset \Omega \setminus \{0\}$ and 
$$f_n = \prod_{j = 1}^{n} \Big(1 - \frac{z}{w_j}\Big), \quad J_{n} = f_{n} - P_{[z f_n]} f_n.$$
Then 
\begin{enumerate}
\item Each $J_n$ is an $S_{\mathscr{H}}$-inner function;
\item the sequence $\|J_n\|$ is a non-decreasing sequence; 
\item $(w_j)_{j \geq 1}$ is a zero sequence for $\mathscr{H}$ if and only if 
$$\sup\{\|J_n\|: n \geq 1\} < \infty.$$
\end{enumerate}
\end{Theorem}

\begin{Example}
 Suppose $\mathscr{H} = H^2$. A result of Takenaka \cite[p.~120]{MR3526203} shows that if $w_j$ are the proposed zeros, then the Gram-Schmidt process applied to the first $n$ Cauchy kernels $k_{w_1}, \ldots, k_{w_n}$ yields 
$$v_1 = \frac{\sqrt{1 - |w_1|^2}}{1 - \overline{w_1} z};$$
$$v_2 = \frac{\sqrt{1 - |w_2|^2}}{1 - \overline{w_2} z} \frac{w_1 - z}{1 - \overline{w_1} z};$$
$$v_3 = \frac{\sqrt{1 - |w_3|^2}}{1 - \overline{w_3} z} \frac{w_1 - z}{1 - \overline{w_1} z}  \frac{w_2 - z}{1 - \overline{w_2} z}; $$
and so on. The condition to be a zero set is then
$$\sup\{\|J_n\|: n \geq 1\} < \infty$$ 
which, by the previous analysis,  translates to 
$$\inf\left\{\|k_0\|^2 - \sum_{j = 1}^{n} |\langle k_0, v_j\rangle|^2: n \geq 1\right\} > 0.$$
A calculation shows that 
$$|\langle k_0, v_j\rangle|^2 = (1 - |w_j|^2) \prod_{i = 1}^{j - 1} |w_i|^2.$$
Furthermore, by telescoping series, 
$$\|k_0\|^2 - \sum_{j = 1}^{n} |\langle k_0, v_j\rangle|^2 = \prod_{j = 1}^{n} |w_j|^2.$$
Thus we have 
$$\inf\left\{\|k_0\|^2 - \sum_{j = 1}^{n} |\langle k_0, v_j\rangle|^2: n \geq 1\right\} = \inf\left\{\prod_{j = 1}^{n} |w_j|^2: n \geq 1\right\}$$
and the above infimum being positive is equivalent to the standard Blaschke condition
\[
       \sum_{j\geq1} \big( 1 - |w_j|  \big) < \infty.
\]
This confirms that the nontrivial zero sets of $H^2$ are exactly the Blaschke sequences.
\end{Example}

\begin{Example}
Let us compute the $S_{\mathscr{H}}$-inner function $J$ corresponding to a one point zero set. Suppose that $\mathscr{H}$ is a reproducing kernel Hilbert space satisfying our assumptions and $$f(z) = 1 - \frac{z}{w}, \quad w \in \Omega \setminus \{0\}.$$
Following the procedure in the derivation of Theorem \ref{v98u24tprefewq23re}, we define 
$$v_{w}(z) = \frac{k_{w}(z)}{\sqrt{k_{w}(w)}},$$ 
the normalized reproducing kernel at $w$. By the formula (\ref{bbbbcbcb}) for $J$ (the inner function corresponding to $f$) we have 
\begin{align}
J & = \frac{k_0 - \langle k_0, v_{w}\rangle v_w}{\|k_0\|^2 - |\langle k_0, v_w\rangle|^2}\nonumber\\
& = \frac{k_0 - \overline{v_w(0)} v_w}{k_{0}(0) - |v_{w}(0)|^2}\nonumber\\
& = \frac{k_0 - \frac{\overline{k_{w}(0)}}{\sqrt{k_{w}(w)}} \frac{k_w}{\sqrt{k_{w}(w)}}}{k_0(0) - \frac{|k_{w}(0)|^2}{k_{w}(w)}}\nonumber\\
& = \frac{k_{w}(w) k_{0} - \overline{k_{w}(0)} k_w}{k_{w}(w) k_{0}(0) - |k_{w}(0)|^2}.\label{zzz29}
\end{align}
Any nonzero constant multiple of an $S_{\mathscr{H}}$-inner function is also $S_{\mathscr{H}}$-inner, and so 
$$k_{w}(w) k_{0} - \overline{k_{w}(0)} k_w$$ is always an $S_{\mathscr{H}}$-inner function. 

 In the $H^2$ case we have 
$$k_{w}(z) = \frac{1}{1 - \overline{w} z}$$ and so \eqref{zzz29} yields 
$$J = \frac{1}{w} \frac{z - w}{1 - \overline{w} z},$$
which, as expected by classical theory, is a constant multiple of a Blaschke factor.

 In the Dirichlet space case, the reproducing kernel is $$k_{w}(z) = \frac{1}{\overline{w} z} \log \Big(\frac{1}{1 - \overline{w} z}\Big)$$ and \eqref{zzz29} yields 
 $$J = \frac{\log(1 - |w|^2)  - \frac{w}{z} \log(1- \overline{w} z)}{\log(1 - |w|^2)  - |w|^2}.$$
 In the Bergman space $\mathscr{B}^2$, we have 
 $$k_{w}(z) = \frac{1}{(1 - \overline{w} z)^2}$$ and \eqref{zzz29} yields 
 $$J = \frac{1 - \frac{(1 - |w|^2)^2}{(1 - \overline{w} z)^2}}{1 - (1 - |w|^2)^2}.$$

 Notice the concept of ``inner'' yields different types of functions in each Hardy, Dirichlet, and Bergman setting.  In the above analysis we see that the expression
\begin{equation} k_{w}(w) k_{0} - \overline{k_{w}(0)} k_w     \label{inneronezero} \end{equation} is always an $S_{\mathscr{H}}$-inner function.  This can also be verified directly from the calculation
$$\langle k_{w}(w) k_{0} - \overline{k_{w}(0)} k_w, S_{\mathscr{H}}^{n}  (k_{w}(w) k_{0} - \overline{k_{w}(0)} k_w)\rangle = 0, \quad n \geq 1.$$
 \end{Example} 
 
These next two results provide an interesting link between the zero set for $\mathscr{H}$ and the property that $\|S_{\mathscr{H}}\|$ or $\|Q_{0}\| = 1$. 

\begin{Theorem}\label{pppxixixi777}
Suppose that $\mathscr{H}$ is a RKHS of anaytic functions on $\mathbb{D}$ satisfying conditions \eqref{P1}--\eqref{P4}. If $\|S_{\mathscr{H}}\|\leq1$, then the union of a zero set with a Blaschke sequence is again a zero set for $\mathscr{H}$.
\end{Theorem}

\begin{proof}  
For notational convenience let $S = S_{\mathscr{H}}$. 
First, suppose that $J$ is $S$-inner, and $w \in \D \setminus \{0\}$.  
Since $J$ is $S$-inner, we have $J \perp S^k J$ for all $k\geq 1$.  Let  
$$F(z) = \sum_{k=0}^{d} F_k z^k$$ be any polynomial with $F_0 = 1$.
By the linearity of $\perp$ (in the second slot) in a Hilbert space, and the Pythagorean Theorem,
\begin{align*}
     \| JF\|^2  &= \| J +  F_1 SJ + F_2 S^2 J + \cdots + F_d S^d J \|^2  \\
       &= \|J\|^2 + \| F_1 SJ + F_2 S^2 J + \cdots + F_d S^d J\|^2\\
       &\leq \|J\|^2  +  \|S\|^2 \| F_1 J + F_2 S  J + \cdots + F_d S^{d-1} J\|^2\\
       &\leq \|J\|^2  +  \|S\|^2 | F_1|^2 \| J\|^2 + \|S\|^2 \| F_2 S  J + \cdots+ F_d S^{d-1} J \|^2\\
       &\leq \cdots\\
       &\leq \|J\|^2 \big( 1 + |F_1|^2 \|S\|^2 + |F_2|^2 \|S\|^{2\cdot 2} + \cdots  + |F_d|^2 \|S\|^{2\cdot d} \big)\\
       &= \|J\|^2 \big( 1 + |F_1|^2  + |F_2|^2  + \cdots  + |F_d|^2  \big).
\end{align*}
The final expression in parentheses is the square of the norm in $\ell^2_A$ of $F(z)$.  The inequality remains true if $F$ is the Blaschke factor that vanishes at $w$, normalized so that $F(0)=1$, i.e., 
$$F(z) = \frac{1}{w} \frac{w - z}{1 - \overline{w} z}.$$
 This function has norm in $\ell^2_A = H^2$ given by $1/|w|$.

Now let $W$ be any zero set for $\mathscr{H}$, and let $\{w_1, w_2, w_3,\ldots\} \in \D \setminus \{0\}$.  Let $J_W$ be the $S$-inner function associated with $W$ with $J(0) = 1$, i.e., $J = f - \widehat{f}$, where $f \in \mathscr{H}$ and $f$ has zeros $W$ (according to multiplicity) and $f(0) = 1$. By repeated application of the above argument, we find that
\[
      \|J_{W \cup \{w_1, w_2,\ldots,w_m\}} \| \leq \frac{\|J_W\| }{|w_1 w_2\cdots w_m|}, \quad m \geq 1.
\]
This, in conjunction with Theorem \ref{v98u24tprefewq23re}, proves the assertion. 
\end{proof}

\begin{Theorem}
Suppose that $\mathscr{H}$ is a RKHS of analytic functions on $\mathbb{D}$ satisfying conditions \eqref{P1}--\eqref{P4}.   If $\|Q_0\| =1$, $J$ is an $S$-inner function with zero set  $W$, and $f \in [J] \setminus \{0\}$, then the zero set for $f$ is the union of $W$ and a Blaschke sequence.
\end{Theorem}

\begin{proof}  
For any $g \in \mathscr{H}$ observe that 
$$g = Q_{0} S g$$ and so, since $\|Q_{0}\| = 1$ by assumption, 
$$\|g\| \leq \|Q_{0}\| \|S g\| = \|S g\|.$$
Apply this identity $k$ times to get 
\begin{equation}\label{1563rtewyufgdhbjk}
\|S^{k} g\| \geq \|g\|, \quad k \geq 1.
\end{equation}
 Suppose that $f \in [J]$. 
By the inner property of $J$, along with repeated use of \eqref{1563rtewyufgdhbjk}, 
\begin{align*}
     \| J F \|^2 &= \| J F_0 \|^2  +  \| S J F_1 + S^2 J F_2 + \cdots \|^2\\
     & = \| J F_0 \|^2  +  \| S (J F_1 + S J F_2 + \cdots) \|^2\\
     & \geq  \| J F_0 \|^2 + \|J F_1\|^2 + \|S J F_2 + \cdots) \|^2\\
      & \vdots\\
    &=  \|J\|^2 \big( {|F_0|^2} + {|F_1|^2} + {|F_2|^2} + \cdots \big)
\end{align*}
for any polynomial $F$.  The bound is true for any sequence of polynomials $F_m$ such that $JF_m$ tends to $f$ in $\mathscr{H}$.  This tells us that $f$ is the product of $J$ and a function in $H^2$.  The claim follows.
\end{proof}


\section{Zeros of $S$-inner functions}\label{mzkqwwowowo}

In the Hardy space $H^2$, we know that when $f \in H^2 \setminus \{0\}$, the classical inner part $I_{f}$ of $f$  takes the form $I_{f} = B S_{\mu}$, where $B$ is the Blaschke product and $S_{\mu}$ is an inner function. The Blaschke factor contains all the zeros of $f$ in $\D$ (and no others) while the inner factor $S_{\mu}$ has no zeros in $\mathbb{D}$. This means that the inner factor $I_f$ has {\em precisely} the same zeros as $f$ (counting multiplicity). How ubiquitous is this phenomenon? In other words, if $\mathscr{H}$ is a Hilbert space of analytic functions satisfying conditions \eqref{P1} - \eqref{P4} and $f \in \mathscr{H}\setminus \{0\}$, does the $S_{\mathscr{H}}$-inner function 
$$J = f - P_{[z f]} f$$ have any ``extra'' zeros inside $\D$? Certainly $J$ has {\em at least} the zeros of $f$. Does it have any others? A result of Hedenmalm and Zhu show that in the weighted Bergman space of analytic functions $f$ on $\D$ for which $f \in L^2((1 - |z|^2)^{\alpha} dA)$, where $dA$ is planar Lebesgue measure, it is possible, when $\alpha > 4$, for the inner function $J$ corresponding to the linear function $f(z) = 1 - z/w$ to have an extra zero in $\D$. So, indeed, the ``no extra zeros'' property for $S_{\mathscr{H}}$-inner functions is not ubiquitous. In this section we obtain lower bounds for these extra zeros and show that they must lie somewhat close to the boundary. Moreover, we will see that in some situations such extra zeros do not exist at all. 

From condition \eqref{P4},  we know that for each $w \in \Omega$, the operator 
$$Q_{w}: \mathscr{H} \to \mathscr{H}, \quad Q_{w} f(z) = \frac{f(z) - f(w)}{z - w}$$ is well defined and continuous. Our criterion that the $S_{\mathscr{H}}$-inner function $J$ has no extra zeros will be stated in terms of the norm of the operator $Q_0$. This operator 
$$(Q_{0} f)(z) = \frac{f(z) - f(0)}{z}$$ is often called the {\em backward shift operator} since if $\Omega = \D$, then $Q_{0}$ acts on the Taylor series of $f$ (about the origin) by shifting all of the coefficients backwards and dropping the constant term, i.e., 
$$Q_{0}(a_0 + a_1 z + a_2 z^2 + \cdots) = a_{1} + a_{2} z + a_{3} z^2 + \cdots.$$

\begin{Theorem}\label{extra1}
   Let $f \in \mathscr{H} \setminus \{0\}$, and let $J = f - P_{[z f]} f$ be the $S_{\mathscr{H}}$-inner function corresponding to $f$.  If $w \in \Omega \setminus \{0\}$ is a zero of $J$ that is not a zero of $f$, then 
\[
     |w| \geq \frac{[1+\|S_{\mathscr{H}}\|^2\|Q_w\|^2]^{1/2}}{\|Q_0\|\|S_{\mathscr{H}}\|\|Q_w\|}.
\]
\end{Theorem}

Towards the proof of this theorem, we start with the following. 

\begin{Proposition}\label{324328t343refd}
Let $f \in \mathscr{H} \setminus \{0\}$ and let $J = f - P_{[z f]} f$.  If $w$ is a zero of $J$ that is not a zero of $f$, then 
$Q_{w} J \in [f]$.
\end{Proposition}

\begin{proof}
By hypothesis, there are polynomials $\phi_n$ such that $\phi_n f$ converges in norm to $J$.  It follows that $Q_w (\phi_n f)$ converges in norm to $Q_w J$, i.e.,
\[
   \frac{\phi_n(z)f(z) - \phi_n(w)f(w)}{z-w} \longrightarrow Q_w J.
\]
Since evaluation at $w$ is bounded, we may further conclude that \[\phi_n(w)f(w) \rightarrow J(w) = 0,\] and hence $Q_w J $ is the limit in norm of
\begin{align*}
    &\quad \frac{\phi_n(z)f(z)-\phi_n(w)f(w)}{z-w}\\
    &=  \frac{\phi_n(z)f(z)-\phi_n(w)f(z)+ \phi_n(w)f(z) -\phi_n(w)f(w)}{z-w} \\
    &=  \frac{\phi_n(z)-\phi_n(w)}{z-w}f(z)+ \frac{f(z) -f(w)}{z-w}\phi_n(w).
\end{align*}
The last term above tends to zero which says that $Q_w J \in [f]$.
\end{proof}

\begin{proof}[Proof of Theorem \ref{extra1}]
Observe that 
\begin{align*}
    \left\|  \frac{J(z)}{1 - \frac{z}{w}} \right\|^2 &=   \left\|  \frac{J(z)}{1 - \frac{z}{w}} \Big( 1 - \frac{z}{w} + \frac{z}{w}  \Big)   \right\|^2   \\
    &=   \left\|  {J(z)}+    \frac{z}{w} \frac{J(z)}{1 - \frac{z}{w}}   \right\|^2.
\end{align*}
Now apply Proposition \ref{324328t343refd} and the Pythagorean Theorem to get
\begin{align*}
    \left\|  \frac{J(z)}{1 - \frac{z}{w}} \right\|^2  &=   \left\|  {J(z)}\right\|^2  +   \left\| \frac{z}{w} \frac{J(z)}{1 - \frac{z}{w}}   \right\|^2   \\
     \left\|  Q_0\Big(\frac{ zJ(z)}{1 - \frac{z}{w}}\Big) \right\|^2  &=   \left\|  {J(z)}\right\|^2  +  \frac{1}{|w|^2} \left\|  \frac{zJ(z)}{1 - \frac{z}{w}}   \right\|^2   \\
      \Big(\|Q_0\|^2 - \frac{1}{|w|^2}\Big) \left\|  \frac{zJ(z)}{1 - \frac{z}{w}} \right\|^2  &\geq   \left\|  {J(z)}\right\|^2   \\
      \Big(\|Q_0\|^2 - \frac{1}{|w|^2}\Big)   \|S_{\mathscr{H}}\|^2 \|Q_w\|^2 |w|^2 \|J(z)\|^2  &\geq   \left\|  {J(z)}\right\|^2   \\
       \Big(\|Q_0\|^2 - \frac{1}{|w|^2}\Big)   \|S_{\mathscr{H}}\|^2 \|Q_w\|^2 |w|^2   &\geq   1   \\
        |w|^2 \geq \frac{1+\|S_{\mathscr{H}}\|^2\|Q_w\|^2}{\|Q_0\|^2\|S_{\mathscr{H}}\|^2\|Q_w\|^2}. \quad \qedhere
\end{align*}
\end{proof}

As a corollary to this theorem we note that if $Q_0$ is contractive, and $\Omega = \mathbb{D}$, then $J$ will have no extra zeros. 

\begin{Corollary}\label{spoifpof}  Let $\mathscr{H}$ be a RKHS of analytic functions on $\mathbb{D}$.
If $Q_0$ is contractive, then the $S_{\mathscr{H}}$-inner function $J$ corresponding to $f$ will have no extra zeros. 
\end{Corollary}

\begin{proof}
If $\|Q_0\| \leq 1$, then
$$\frac{[1+\|S_{\mathscr{H}}\|^2\|Q_w\|^2]^{1/2}}{\|Q_0\|\|S_{\mathscr{H}}\|\|Q_w\|} \geq \frac{[1+\|S_{\mathscr{H}}\|^2\|Q_w\|^2]^{1/2}}{\|S_{\mathscr{H}}\|\|Q_w\|} \geq 1.$$
By Theorem \ref{extra1}, any extra zero $w \in \D$ must satisfy $|w| \geq 1$. 
\end{proof}

It is easy to see that for the Hardy space $H^2$, the operator $Q_{0}$ (which is just the well-known backward shift operator) satisfies $\|Q_{0}\| = 1$ and so the $S$-inner function $J$ of corresponding to $f$, which in this case is the classical inner factor of $f$, never has  extra zeros. Slightly more work is that on the Dirichlet space $\mathscr{D}$ (See Example \ref{00uuucbbjbbjb}), the operator $Q_0$ also has norm equal to one \cite{MR0358396}. This gives us the following. 

\begin{Corollary}\label{77w9e8r7ew98}
For any $f \in \mathscr{D}$, the corresponding $S_{\mathscr{D}}$-inner function $J$ has no extra zeros in $\D$. 
\end{Corollary}

We point out here that this result, in a way, is known. As shown in \cite{MR1259923}, every shift invariant subspace $M$ of the Dirichlet space has the property that $M \ominus z M = \C \phi$ and this function generates $M$, in that $\bigvee\{\phi, z \phi, z^2 \phi, \ldots\} = M$. Applying this fact to a vector $f \in \mathscr{D}$ and $M = [f]$, we see that $[J] = [f]$ and so $J$ cannot have any extra zeros. 

For the Bergman space $\mathscr{B}$ from Example \ref{popoweripeo}, $Q_0$ has norm $\sqrt{2}$ and so we are unable to apply Corollary \ref{spoifpof}. However, it is known, for different reasons \cite{MR1440934}, that $J$ has no extra zeros. On the other hand, for the space $H^{2}_{1}$ from Example \ref{erewrwe}, one can quickly check (using power series) that $Q_{0}$ is contractive on $H^{2}_{1}$ and thus we have the following.

\begin{Corollary}\label{xhbmdfer4terweqwa}
For any $f \in H^{2}_{1}$, the corresponding inner function $J$ has no extra zeros in $\D$.
\end{Corollary}


\section{Extra Zeros Abound}\label{ExtraZ}

In the previous section it was shown that if an $S$-inner function $J$ corresponding to a given function $f$ has extra zeros, then those extra zeros must be bounded away from the origin.  When $\Omega = \mathbb{D}$, this gave rise to a sufficient condition on the space $\mathscr{H}$ for the $S$-inner functions to have no extra zeros.  In the present section we shall see that extra zeros are nonetheless quite abundant.  A large class of spaces will be constructed for which certain $S$-inner functions will have extra zeros.   

We begin by presenting another description of the zero sets for a RKHS $\mathscr{H}$ satisfying our hypotheses.  This description is due to Shapiro and Shields \cite{MR0145082}.

Let $W_n := \{w_1, w_2,\ldots,w_n\}, w_j \in \Omega \setminus \{0\},$
and define
$$
     f_n(z) = \Big( 1 - \frac{z}{w_1}  \Big) \Big( 1 - \frac{z}{w_2}  \Big)\cdots  \Big( 1 - \frac{z}{w_n}  \Big),$$
     and
     $J_n = f_n - P_{[z f_n]} f_n$.
For notational simplicity, let $k_j$ be the reproducing kernel for $w_j$ and let $k_0$ be the reproducing kernel at the origin.
 
From \eqref{bbbbcbcb} we know that $J_n$ has the representation
\begin{equation}\label{repinnereq}
     J_n(z)  =  c_{n,0} k_0 + c_{n,1} k_1 + c_{n,2} k_2 +\cdots+c_{n,n} k_n, 
\end{equation}
 where the coefficients $c_{n,j}$ are uniquely determined by the conditions 
 $$J_n(w_1)=J_n(w_2)=\cdots =J_n(w_n)=0, \quad J_n(0)=1.$$  Indeed, the coefficients are the unique solutions to the matrix equation
 \[
     \left[\begin{array}{ccccc} G_{0,0} & G_{0,1} &  G_{0,2} & \ldots & G_{0,n}\\
         G_{1,0} & G_{1,1} &  G_{1,2} & \ldots & G_{1,n}\\
         G_{2,0} & G_{2,1} &  G_{2,2} & \ldots & G_{2,n}\\
         \vdots & \vdots & \vdots & \vdots & \vdots \\
         G_{n,0} & G_{n,1} &  G_{n,2} & \ldots & G_{n,n}
         \end{array}\right] \\
         \left[\begin{array}{c} c_{n,0} \\ c_{n,1} \\ c_{n,2} \\ \vdots \\ c_{n,n}   \end{array}\right]
= \left[\begin{array}{c} 1 \\ 0 \\ 0 \\ \vdots \\ 0  \end{array}\right],
 \]
 where 
 \[
      G = G^{(n)} = G[k_0, k_1, k_2,\ldots,k_n]
 \]
 is the Gramian matrix for the vectors $k_0, k_1, k_2,\ldots,k_n$, and 
 $$G_{s,t} := \langle k_t, k_s  \rangle.$$  Since a finite set of reproducing kernels is linearly independent, the Gramian determinant is nonzero, and hence the matrix $G$ is invertible, guaranteeing a unique solution for the coefficients.
 
Continuing from the above equation, we can write
\begin{align*}
      J_n(z) &= \left[\begin{array}{ccccc} k_0(z) & k_1(z) & k_2(z) & \cdots & k_n(z)\end{array}  \right] \left[\begin{array}{c} c_{n,0} \\ c_{n,1} \\ c_{n,2} \\ \ldots \\ c_{n,n}   \end{array}\right] \\
      &= \left[\begin{array}{ccccc} k_0(z) & k_1(z) & k_2(z) & \cdots & k_n(z)\end{array}  \right] {G^{(n)}}^{-1} \left[\begin{array}{c} 1 \\ 0 \\ 0 \\ \ldots \\ 0  \end{array}\right]\\
      &=\big[A_{0,0}k_0(z) - A_{0,1}k_1(z) +  \cdots+ (-1)^n A_{0,n} k_n(z)\big]/\det G^{(n)},
\end{align*}
where $A_{m,n}$ is the $(m,n)$th cofactor of $G^{(n)}$.  But the last quantity in square brackets is itself the determinant of a certain matrix, yielding
\begin{equation}\label{detrepinnerfcneq}
     J_n(z) \det G^{(n)} =  \det \left[\begin{array}{ccccc} k_0(z) & k_1(z) &  k_2(z) & \ldots & k_n(z)\\
         G_{1,0} & G_{1,1} &  G_{1,2} & \ldots & G_{1,n}\\
         G_{2,0} & G_{2,1} &  G_{2,2} & \ldots & G_{2,n}\\
         \ldots & \ldots & \ldots & \ldots & \ldots \\
         G_{n,0} & G_{n,1} &  G_{n,2} & \ldots & G_{n,n}
         \end{array}\right]
\end{equation}

Let 
\begin{equation}\label{dn0000}
       d_n := \inf \big\| k_0 - (c_1 k_1 + c_2 k_2 + \cdots c_n k_n) \big\|
\end{equation}
where the infimum is over the coefficients $c_1, c_2,\ldots, c_n$.  It is well known that 
\begin{equation}\label{detratioeq}
     d_n^2 = \frac{\det G[k_0, k_1, k_2,\ldots,k_n]}{\det G[k_1, k_2,\ldots,k_n]}.
\end{equation}
A proof of this appears in \cite[Lemma 4.2.4]{EKMR}.

Furthermore, Oppenheim's inequality (see, for example, \cite{HJ}) tells us that for nonnegative definite square matrices $A = (a_{s, t})$ and $B = (b_{s, t})$, the Hadamard product of $A$ and $B$, i.e, $(a_{s, t} b_{s, t})$, satisfies
\begin{equation}\label{detprodposmateq}
     \det [a_{s,t}b_{s,t}] \geq \Big(\det [a_{s,t}] \Big)\Big(\prod_{t} b_{t,t}\Big).
\end{equation}
This enables us to derive, as was done in \cite{MR0145082}, the following sufficient condition for a zero set of $\mathscr{H}$ (see also \cite{EKMR} for an exposition of this).

\begin{Theorem}\label{zerosetcritprop}
    Let $\{w_1, w_2, w_3,\ldots\} \subset \D \setminus \{0\}$ be a sequence of distinct points.   If the matrix
    \begin{equation}\label{nonnegcond}
        \big[ 1 -  k_0(w_s) \overline{k_0(w_t)}/k_{w_t}(w_s)k_0(0)\big]_{1\leq s, t\leq n}
    \end{equation}
    is nonnegative definite for all $n\geq 1$,
    and
    \begin{equation}\label{suffcondzereq}
         \inf_{n} \prod_{m=1}^{n} \left[1 - \frac{|k_0(w_m)|^2} {k_{w_m}(w_m)k_0(0)}\right]  > 0,
    \end{equation}
    then there exists a nonzero $f\in \mathscr{H}$ such that $f(w_n)=0$ for all $n\geq 1$.
\end{Theorem}

\begin{proof}
    By (\ref{skdfhgsdkfld11}) and Theorem \ref{v98u24tprefewq23re}, it is enough to show that the quantity $d_n$ from \eqref{dn0000} satisfies  $\inf d_n > 0$.  Let us examine $\det G^{(n)}$, with a view towards applying (\ref{detratioeq}).  This determinant is unchanged if the multiple of any row is added to a different row.  Suppose that $G_{s,0}/G_{0,0}$ times the $0$th row (the rows and columns are indexed from $0$ to $n$) is added to the $s$th row, for all $1\leq s \leq n$.  The result is that
    \begin{align*}
          \det G^{(n)}  &=  G_{0,0} \det \big[ G_{s,t} - G_{s,0}G_{0,t}/G_{0,0} \big]_{1\leq s, t\leq n}\\
          &=   G_{0,0} \det \big[ G_{s,t} (1 - G_{s,0}G_{0,t}/G_{s,t}G_{0,0}) \big]_{1\leq s, t\leq n}\\
          &\geq G_{0,0} \Big(\det G[k_1, k_2,\ldots,k_n] \Big)\Big( \prod_{t=1}^{n} [1 - G_{t,0}G_{0,t}/G_{t,t}G_{0,0}]  \Big),
    \end{align*}
where in the last step we applied  (\ref{detprodposmateq}).  The claim now follows from invoking (\ref{detratioeq}), and writing out $G_{t,t}$ in terms of the kernel functions.
\end{proof}

\begin{Example}
     When $\mathscr{H} = H^2$, the matrix in (\ref{nonnegcond}) takes the form
     \[
            \left[  \bar{w}_t w_s \right]_{1\leq s, t\leq n},
     \]
     which is obviously positive definite.
     
     Now the zero set criterion (\ref{suffcondzereq}) is 
     \begin{align*}
        0 &< \inf_{n} \prod_{m=1}^{n} \Big[1 - \frac{|k_0(w_m)|^2} {k_{w_m}(w_m)k_0(0)}\Big]  \\
           &= \inf_{n}\prod_{m=1}^{n} [1 - \frac{1}{1/(1 - |w_m|^2)}] \\
           &= \inf_{n}\prod_{m=1}^{n} |w_m|^2.
     \end{align*}
     Of course, this is equivalent to the Blaschke condition.
\end{Example}

In \cite{MR0145082} the zero sets of functions in the Dirichlet space $\mathscr{D}$ (and other related spaces) were discussed.  The Dirichlet space can be viewed as the weighted $\ell^2$ space with weights $1, 2, 3,\ldots, n+1,\ldots$.  We now construct a large class of such weighted spaces for which the corresponding matrices \ref{nonnegcond} are nonnegative definite, and hence lie within the scope of Theorem \ref{zerosetcritprop}.

\begin{Example}\label{weispell2lmaex} 
Fix $\Omega = \mathbb{D}$, and let $w_1, w_2, w_3,\ldots$ be a sequence of distinct nonzero points in $\mathbb{D}$.  
Suppose that $\Lambda := (\lambda_n)_{n\geq 0}$ is a sequence of positive numbers with $\lambda_0=1$, and define $\ell^2(\Lambda)$ to be the Hilbert space of sequences $f = (f_n)_{n\geq 0}$ such that
\[
     \|f\| = \Big( \sum_{n=0}^{\infty} |f_n|^2 \lambda_n  \Big)^{1/2} < \infty.
\]
Provided that the weights $\lambda_n$ do not decay to zero too rapidly, each member of $\ell^2(\Lambda)$ can be identified with the analytic function
\[
    f(z) = \sum_{n=0}^{\infty} f_n z^n
\]
on $\mathbb{D}$.   (For example, if the weights decay exponentially, then $\ell^2(\Lambda)$ will contain some coefficient sequences that increase exponentially; such functions will not necessarily be analytic in all of $\mathbb{D}$.)
The reproducing kernel function
\[
      k_{w}(z) :=  \sum_{n=0}^{\infty}  \frac{(\bar{w}z)^n}{\lambda_n} 
\]
implements point evaluation at $w \in \mathbb{D}$.  Again, if the weights $\lambda_n$ do not decay too rapidly, the kernel function will be analytic in $\mathbb{D}$.   Notice that point evaluation at the origin corresponds to the constant kernel $1$.

Let us determine sufficient conditions on the sequence $\Lambda$ of weights for the matrix in (\ref{nonnegcond}) to be nonnegative definite.
We claim that for any $a>0$, and positive integers $m$ and $n$, the matrix
\[
           M :=  \left[  a(\bar{w}_t w_s)^m \right]_{1\leq s, t\leq n}
\]
is nonnegative definite.  This is because 
\[
      C^* M C =   a \big| c_1 w_1^m + c_2 w_2^m + \cdots + c_n w_n^m \big|^2  \geq 0
\]
for any column vector $C$ with $C^* = [ \bar{c}_1\ \bar{c}_2\ \ldots \bar{c}_n]$.  For $n$ fixed the sum of any such matrices is also nonnegative definite.  In particular, if $(a_m)_{m\geq 1}$ is a sequence of nonegative numbers with $a_1 > 0$ and $\sum_{m=1}^{\infty} a_m  \leq 1$, the matrix
\[
        \left[ \sum_{m=1}^{\infty}    a_m (\bar{w}_t w_s)^m \right]_{1\leq s, t\leq n}
\]
is nonnegative definite.

It is clear that the function of $z$ defined by
\[
       \Phi(z) :=  \frac{1}{ 1 - \sum_{m=1}^{\infty}    a_n  z^m }
\]
is analytic in $\mathbb{D}$, and has a convergent power series
\begin{equation}\label{kerfcnindisg}
      \Phi(z) =  1 + \sum_{n=1}^{\infty} b_n z^n
\end{equation}
in $\mathbb{D}$.  By expressing $\Phi$ as the geometric series
\[
    \Phi(z) = 1 + \Big(\sum_{m=1}^{\infty}    a_n  z^m\Big) + \Big(\sum_{m=1}^{\infty}    a_n  z^m\Big)^2 +\cdots,
\]
 and using the assumption that $a_1 >0$, we find that each $b_n$ is positive (see also the Kaluza lemma \cite[p. 69]{EKMR}).

Thus, with the identification $b_n = 1/\lambda_n$ for all $n\geq 1$, the function $\Phi(\bar{w}z)$ is the reproducing kernel in the weighted space $\ell^2(\Lambda)$ for $w\in\mathbb{D}$.
It then follows that
\begin{align*}
       1 -  k_0(w_s) \overline{k_0(w_t)}/k_{w_t}(w_s)k_0(0)
       &=      1 -  1/k_{w_t}(w_s) \\
       &=     1 -  \Big(1 + \sum_{n=1}^{\infty} b_n (\bar{w}_t w_s)^n\Big)^{-1} \\
       &=     1 -  \Big(1 - \sum_{m=1}^{\infty}    a_n (\bar{w}_t w_s)^m\Big) \\
       &=     \sum_{m=1}^{\infty}    a_n (\bar{w}_t w_s)^m\Big. 
\end{align*}
That is to say, for the weighted space $\ell^2(\Lambda)$, the matrix in (\ref{nonnegcond}) is nonnegative definite.  

According to Theorem \ref{zerosetcritprop}, a sequence $w_1, w_2, w_3,\ldots$ of distinct nonzero points of $\mathbb{D}$ is the zero set of some nontrivial function $f\in\ell^2(\Lambda)$ if 
\[
     \inf_{n\geq 1} \prod_{m=1}^{n} \left[1 - \frac{1} {1 - \sum_{j=1}^{\infty} a_j |w_m|^{2j}}\right]  > 0.
\]
This provides a sufficient condition for a zero set of $\ell^2(\Lambda)$. \label{previousexample}
\end{Example}

\begin{Example}  With the definitions of Example \ref{previousexample}, it was shown in \cite{MR0145082} that if the sequence $(b_n)_{n\geq 0}$ (the reciprocals of the weights of the space $\ell^2(\Lambda)$) satisfies 
\[
      b_n^{2} \leq b_{n+1}b_{n-1}
\]
for all $n\geq 1$, then the matrices given by \ref{nonnegcond} are nonnegative definite, and thus Theorem \ref{zerosetcritprop} applies.  This class of examples  includes the Dirichlet space $\mathscr{D}$.
\end{Example}

Here is another way to see how extra zeros may arise. Recall the formula from \eqref{repinnereq}
$$J_n(z)  =  c_{n,0} k_0 + c_{n,1} k_1 + c_{n,2} k_2 +\cdots+c_{n,n} k_n$$
for expressing the $S_{\mathscr{H}}$-inner function of a finite zero set in terms of the corresponding kernel functions.
 
\begin{Lemma}The $S_{\mathscr{H}}$-inner function $J_{n - 1}$ has an extra zero at the point $w_n$ if and only if the coefficient $c_{n,n}$ vanishes.
\end{Lemma}

\begin{proof} Suppose that $c_{n,n} = 0$.  Then $J_n$ has the following properties: 
 $$J_n(0) = 1, J_n(w_1)=\cdots =J_n(w_{n-1})=0,$$ and
$$\langle z^m f_{n-1}, J_n \rangle = 0, \quad m\geq 1,$$  This forces the identification $J_n = J_{n-1}$.  Since $J_{n-1}(w_n)=0$, it can be said that $w_n$ is an extra zero of $J_{n-1}$.

Conversely suppose that $w_n$ is an extra zero of $J_{n-1}$.  First, for any $m\geq 1$,
\begin{align*}
    \langle z^m f_n(z), J_{n-1} \rangle &= \langle z^m f_{n-1}(z), J_{n-1} \rangle - \langle z^{m+1}(1/w_n) f_{n-1}(z), J_{n-1} \rangle \\
    &= 0.
\end{align*}
Furthermore, 
$$J_{n-1}(w_1)=J_{n-1}(w_2)=\cdots =J_{n-1}(w_n)=0, \quad J_{n-1}(0)=1.$$  This implies that $J_{n-1} = J_n$.  Since the representations 
(\ref{repinnereq}) are unique, it must be that $c_{n,n}=0$.  
\end{proof}

Let us calculate $c_{n,n}$.  Let $H$ be the matrix $G^{(n)}$ with its $n$th column (the columns are indexed 0 through $n$) replaced by
$[1\ 0\ 0\ \ldots\ 0]^{T}$.  By Cramer's Rule,
\[
       c_{n,n} = \frac{\det H}{\det G^{(n)}}.
\]
Since the last column of $H$ is such a special form, taking the determinant of $H$ results in $(-1)^n$ times the determinant of the following submatrix of $G^{(N)}$:
\[
    R:=  \left[\begin{array}{ccccc} 
         G_{1,0} & G_{1,1} &  G_{1,2} & \ldots & G_{1,n-1}\\
         G_{2,0} & G_{2,1} &  G_{2,2} & \ldots & G_{2,n-1}\\
         \ldots & \ldots & \ldots & \ldots & \ldots \\
         G_{n,0} & G_{n,1} &  G_{n,2} & \ldots & G_{n,n-1}
         \end{array}\right].
\]
\begin{Proposition}\label{linearkernelprop}
   The inner function $J_{n-1}$ corresponding to the zero set $w_1, w_2,\ldots, w_{n-1}$ has an extra zero at $w_n$ precisely when $\det R =0$.
\end{Proposition}

By comparing this situation to the representation (\ref{detrepinnerfcneq}), we can confirm that this is a way of expressing $J_{n-1}(w_n)=0$.

When $n=2$ this gives us a simple criterion for deciding whether the inner function corresponding to a linear polynomial has an extra zero.  In this situation,
\[
     \det R = \det \left[\begin{array}{cc} 
         G_{1,0} & G_{1,1} \\
         G_{2,0} & G_{2,1}
         \end{array}\right]
         = \langle k_0, k_1 \rangle \langle k_2, k_1 \rangle - \langle k_0, k_2 \rangle \langle k_1, k_1 \rangle .
\]
Thus by another route we have arrived at the inner function identified in (\ref{inneronezero}).

\begin{Example}
Consider the case $\mathscr{H} = H^2$.  The $S$-inner functions are the classical inner functions, which have no extra zeros.  Let us confirm this for linear polynomials, using Proposition \ref{linearkernelprop}.  Let $r$ and $s$ be distinct nonzero points in $\mathbb{D}$.  Then the inner part of the linear polynomial
\[
     f(z) = 1 - \frac{z}{r}
\]
has the extra zero $s$ precisely if
\[
        1 \cdot \frac{1}{1 - \bar{s}r} = 1 \cdot \frac{1}{1-|r|^2}.
\]
Of course, this never happens when $r\neq s$, reflecting that the Blaschke factor vanishing at $r$ vanishes nowhere else.  To rule out the possibility of a double root at $r$, we use the kernel function
\[
   k_{1,r} = \frac{1}{(1-\bar{r}z)^2},
\]
for evaluation of a derivative at $r$.  The criterion then becomes
\[
     1 \cdot \frac{1}{(1-|r|^2)^2} = 1 \cdot \frac{1}{1 - |r|^2},
\]
which is also impossible.
\end{Example}

Finally, we demonstrate that there are numerous spaces for which there exist $S$-inner functions with extra zeros.

\begin{Example}  
    Let us return to the weighted spaces $\ell^2(\Lambda)$ of Example \ref{weispell2lmaex}, and consider the special case that the weights arise in connection with the choice
    \[
        \Phi(z) = \frac{1}{1 - a_1 z - a_2 z^2},
    \]
    where $a_1 + a_2 \leq 1$, and $a_2 > 4a_1 > 0$.  Then, by use of the geometric series formula we find that
    \[
         \Phi(z) =  1 + \sum_{n=1}^{\infty} b_n z^n,
    \]
    with
    \begin{align*}
      b_{2n-1} &= {2n-1 \choose 0} a_1^{2n-1} + {2n-2 \choose 1}a_1^{2n-3}a_2 +{2n-3 \choose 2}a_1^{2n-5}a_2^2 \\
                     &\quad + \cdots + {n \choose n-1} a_1 a_2^{n-1}\\
                     &\leq (a_1^2 + a_2)^{2n}/a_1\\
                     &\leq 1/a_1\\
      b_{2n} &= {2n \choose 0} a_1^{2n} + {2n-1 \choose 1}a_1^{2n-2}a_2 +{2n-2 \choose 2}a_1^{2n-4}a_2^2 \\
                     &\quad + \cdots + {n \choose n}  a_2^{n}\\
                     &\leq (a_1^2 + a_2)^{2n}\\
                     &\leq 1.
    \end{align*}
    for all $n\geq 1$.
    
    Each coefficient $b_n$ is positive, and so we may define the weights $\lambda_0 = 1$, and $\lambda_n = 1/b_n$, $n\geq 1$.   The weights are bounded away from zero, and therefore the functions belonging to $\ell^2(\Lambda)$ are analytic in $\mathbb{D}$.
   Furthermore, point evaluation at $w\in\mathbb{D}$ arises from the reproducing kernel function
    \[
           k_w(z) = \Phi(\bar{w}z) = \frac{1}{1-a_1\bar{w}z - a_2 (\bar{w}z)^2},
    \]
    which is obviously analytic in $\mathbb{D}$.
    
    The inner function associated with the polynomial $1 - z/w$ has an extra zero $\zeta$, distinct from $w$, provided that
    \begin{align*}
             \langle k_0, k_w \rangle \langle k_{\zeta}, k_w \rangle &= \langle k_0, k_{\zeta} \rangle \langle k_w, k_w \rangle\\
             1 \cdot \frac{1}{1-a_1\bar{w}\zeta - a_2 (\bar{w}\zeta)^2} &= 1 \cdot \frac{1}{1-a_1\bar{w}w - a_2 (\bar{w}w)^2}\\
             {a_1\bar{w}(\zeta-w) + a_2 \bar{w}^2(\zeta-w)^2} &= 0\\
             {a_1\bar{w} + a_2 \bar{w}^2(\zeta+w)} &= 0\\
             \zeta &= -\frac{a_1  + a_2 |w|^2}{a_2\bar{w}}
    \end{align*}
   But by assumption $a_2 > 4a_1$, so we can choose $w\in\mathbb{D}$ so that 
     \[
          |w| - |w|^2 > a_1/a_2,
     \]
     which in turn implies that $\zeta \in \mathbb{D}$.
     
     We have thus constructed a family of spaces $\mathscr{H}$ of analytic functions on $\mathbb{D}$ for which there exist $S$-inner functions having extra zeros.  This shows that the phenomenon of extra zeros is in some way unexceptional.  
\end{Example}


\section{Inner vectors in Banach spaces}\label{section7}

Recall from Section \ref{s2} that a vector $\vec{v}$ in a Hilbert space is $T$-inner if 
\begin{equation}\label{deftinner}
\langle \vec{v}, T^n \vec{v}\rangle = 0,\ n \geq 1.
\end{equation}
We want to extend the definition of $T$-inner vectors to Banach spaces.  
However, first we need a notion of ``orthogonality'' so we can make sense of the very definition in a Banach space. Indeed, what do we mean by $ \vec{v} \perp T^n \vec{v}$ when there is no inner product?

Before jumping into our definition of orthogonality, we  need to review a few necessary facts. See \cite{Car} for the details. For a complex Banach space $\mathscr{X}$ with norm $\|\cdot\|$, we say that $\mathscr{X}$ is {\em smooth} if given any $\vec{x} \in \mathscr{X} \setminus \{\vec{0}\}$ there is a {\em unique} $\ell \in \mathscr{X}^{*}$ (the norm dual space of $\mathscr{X}$) such that $\|\ell\| = 1$ and $\ell(\vec{x}) = \|\vec{x}\|$. Though not relevant to our discussion here, there is an equivalent definition of smoothness of a Banach space involving the G\^{a}teaux derivative of the norm. It is important to point out that the Hahn-Banach theorem yields the existence of a {\em norming functional} $\ell_{\vec{x}}$ for each $\vec{x} \in \mathscr{X}$. The {\em uniqueness} of the above norming functional for every $\vec{x} \in \mathscr{X} \setminus \{\mathbf{0}\}$ is what makes $\mathscr{X}$ smooth. Hilbert spaces are smooth, as are the Lebesgue spaces $L^{p}(X, \mu)$ when $p \in (1, \infty)$. The spaces $L^1(X, \mu)$ and $L^{\infty}(X, \mu)$ are not smooth. 

A Banach space $\mathscr{X}$ is {\em uniformly convex} if given $\epsilon > 0$, there is a $\delta = \delta(\epsilon) > 0$ such that 
$$\|\vec{x}\| \leq 1, \|\vec{y}\| \leq 1, \|\vec{x} - \vec{y}\| \geq \epsilon \implies \|\tfrac{1}{2}(\vec{x} + \vec{y})\| \leq 1 - \delta.$$
A Hilbert space is uniformly convex and Clarkson's inequalities imply that $L^{p}(X, \mu)$ is uniformly convex when $p \in (1, \infty)$ \cite[page 107]{Car}. A uniformly convex Banach space turns out to be reflexive. Important to this paper is the fact that uniformly convex spaces enjoy the {\em unique nearest point property} in that for a closed subspace (or more generally a closed convex set) $Y$ of $\mathscr{X}$ and a vector $\vec{x} \in \mathscr{X}$, there is a unique vector $\widehat{\vec{x}}\in Y$ for which 
\begin{equation}\label{bbbbbx}
\|\vec{x} - \widehat{\vec{x}}\| \leq \|\vec{x} - \vec{y}\|, \quad \vec{y} \in Y.
\end{equation}
This unique nearest point $\widehat{\vec{x}}$ is called the {\em metric projection} of $\vec{x}$ onto $Y$. When $\mathscr{X}$ is a Hilbert space, $\widehat{\vec{x}}$ turns out to be the orthogonal projection of $\vec{x}$ onto $Y$ and the mapping $\vec{x} \mapsto \widehat{\vec{x}}$ is linear. For a general Banach space, the mapping $\vec{x} \mapsto \widehat{\vec{x}}$ is not necessarily linear. 

We  now follow \cite{AMW, Jam}  and define what it means for vectors to be ``orthogonal'' in a Banach space.  For vectors $\mathbf{x}$ and $\mathbf{y}$ in a Banach space $\mathscr{X}$ we  say that $\mathbf{x}$ is {\em orthogonal to $\mathbf{y}$ in the Birkhoff-James sense} if
\begin{equation}\label{2837eiywufh[wpofjk}
      \|  \mathbf{x} + \beta \mathbf{y} \|_{\mathscr{X}} \geq \|\mathbf{x}\|_{\mathscr{X}}
\end{equation}
for all $\beta \in \C$. In this situation we write $\mathbf{x} \perp_{\mathscr{X}} \mathbf{y}$. A little exercise will show that if $\mathscr{X}$ is a Hilbert space, then $\mathbf{x} \perp \mathbf{y} \iff \mathbf{x} \perp_{\mathscr{X}} \mathbf{y}$. In this generality the relation $\perp_{\mathscr{X}}$ is generally neither symmetric nor linear in either argument. However, in a {\em smooth} Banach space, the relation $\perp_{\mathscr{X}}$ is linear in its second slot, meaning that 
$$\vec{x} \perp_{\mathscr{X}} \vec{y},\ \, \vec{x} \perp_{\mathscr{X}} \vec{z}, \implies \vec{x} \perp_{\mathscr{X}} (\alpha \vec{y} + \beta \vec{z}), \quad \alpha, \beta \in \C.$$
See \cite{Jam} for a proof of this. 

When $\mathscr{X}$ is a smooth Banach space and $\vec{x} \in \mathscr{X}$, we let $\ell_{\vec{x}} \in \mathscr{X}^{*}$ denote the unique norming functional for $\vec{x}$ (recall $\ell_{\vec{x}}(\vec{x}) = \|x\|$). By \cite[Cor.~4.2]{AMW}, we can state Birkhoff-James orthogonality equivalently as 
\begin{equation}\label{xxxkldfkdlk}
\vec{x} \perp_{\mathscr{X}} \vec{y} \iff \ell_{\vec{x}}(\vec{y}) = 0.
\end{equation}

Important to our discussion is the more tangible condition for Birkhoff-James orthogonality in $L^p(X, \mu)$ spaces (see \cite{Jam}): For $f, g \in L^p(X, \mu)$,
\begin{equation}\label{orutreee}
f \perp_{L^p(X, \mu)} g \iff \int_{X} |f|^{p - 2} \overline{f} g d \mu = 0.
\end{equation}
In the above integral, we interpret any instance of $|0|^{p - 2} 0$ to be zero. We have used Birkhoff-James orthogonality in several recent papers to discuss problems involving the $\ell^{p}_{A}$ spaces of analytic functions whose power series coefficients belong to the sequence space $\ell^{p}$. In \cite{MR3686895} we use this orthogonality to give some new bounds on the zeros of an analytic function while in \cite{zeroslp} we use this orthogonality, and the concept of an $\ell^{p}_{A}$-inner function, to describe the zeros sets of $\ell^{p}_{A}$. Still further, we use orthogonality in \cite{CR2} to give a factorization theorem for $\ell^{p}_{A}$ functions. Though perhaps not using explicitly, by name, the authors in \cite{MR1398090} use the above orthogonality to discuss zero sets, via extremal functions, for the $L^p$ Bergman spaces. 

With these preliminary remarks, we are ready to define a notion of inner elements. We make the following assumption for the rest of the paper:
$$\mbox{$\mathscr{X}$ is a uniformly convex, smooth, complex Banach space.}$$ For a bounded linear transformation $T: \mathscr{X} \to \mathscr{X}$ and a nonzero vector $\vec{x} \in \mathscr{X}$, we say that $\vec{x}$ is {\em $T$-inner} when 
$$\vec{x} \perp_{\mathscr{X}} T^n \vec{x}, \quad n \geq 1.$$
By the linearity of the relation $\perp_{\mathscr{X}}$ in the second slot (which follows from our assumptions on $\mathscr{X}$), we see that $\vec{x}$ is $T$-inner if and only if $\vec{x} \perp_{\mathscr{X}} [T \vec{x}]$, where, as a reminder, 
$$[T \vec{x}] = \bigvee\{T \vec{x}, T^2 \vec{x}, T^3 \vec{x}, \ldots\}.$$

If we let  $\widehat{\vec{x}}$ denote the metric projection (nearest point) of $\vec{x}$ onto the subspace $[T \vec{x}]$, equivalently, $\widehat{\mathbf{x}}$ is the unique vector satisfying 
$$\|\vec{x} - \widehat{\vec{x}}\| \leq \|\vec{x} - \vec{y}\|, \quad \vec{y} \in [T \vec{x}],$$
 the proof of Proposition \ref{66zczxc} yields the following. 
 
 \begin{Proposition}\label{v04rkjbfvj}
 If $T$ is a bounded linear transformation on $\mathscr{X}$ and $\vec{x} \in \mathscr{X}$, then 
 $\vec{x} - \widehat{\vec{x}}$ is $T$-inner (or zero) and every $T$-inner vector arises in this manner. 
 \end{Proposition}
 
 Recall that if $T$ is a bounded linear transformation on $\mathscr{X}$, then the Banach space adjoint operator $T^{*}$, i.e., $(T^{*} \ell)(\mathbf{x}) = \ell(T \mathbf{x})$ for all $\mathbf{x} \in \mathscr{X}$ and $\ell \in \mathscr{X}^{*}$, is a bounded linear transformation on $\mathscr{X}^{*}$. 
 
 \begin{Proposition}\label{nnnnuuhuhuh}
Suppose that $T$ is a bounded linear transformation on $\mathscr{X}$. If $\vec{x} \in \mathscr{X}$  is $T$-inner, and $\ell_{\vec{x}}$ is the unique norming functional of $\vec{x}$, then $\ell_{\vec{x}}$ is $T^*$-inner in $\mathscr{X}^*$.
\end{Proposition}

\begin{proof}
The assumption of uniform smoothness implies that each nonzero element of $\mathscr{X}$ has a unique norming functional.  The hypotheses further imply that $\mathscr{X}$ is reflexive and that $\mathscr{X}^*$ is strictly convex and smooth \cite{Car}.  Therefore we may speak of unique norming functionals for both $\mathscr{X}$ and $\mathscr{X}^*$.

Suppose that $\|\vec{x}\|=1$, and 
\[
\vec{x} \perp  T^k \vec{x}, \quad k \geq 1.
\]
This implies that 
$\ell_{\vec{x}}(T^k \vec{x}) = 0$ and 
${T^*}^k \ell_{\vec{x}}(\vec{x}) = 0$
for all $k\geq 1$.  Note that $\vec{x}/\|\vec{x}\|$ can be viewed as the norming functional for $\ell_{\vec{x}}$, since it has norm $1$ 
and
\[
\ell_{\vec{x}}(\vec{x}/\|\vec{x}\|) = \frac{1}{\|\vec{x}\|} \|\vec{x}\| = 1 = \|\ell_{\vec{x}}\|.
\]
It follows that
$\ell_{\vec{x}} \perp {T^*}^k \ell_{\vec{x}}$
for all $k \geq 1$. This says that the vector $\ell_{\vec{x}} \in \mathscr{X}^{*}$ is $T^*$-inner. 
\end{proof}

\begin{Example}
For the Hardy spaces $H^p$, $1<p<\infty$, which we can regard as closed subspaces of $L^p(d \theta/2 \pi)$, we can use \eqref{orutreee} to see that 
$$f \perp_{H^p} g \iff \int_{0}^{2 \pi} |f(e^{i \theta})|^{p - 2} \overline{f(e^{i \theta})} g(e^{i \theta}) \frac{d \theta}{2 \pi} = 0.$$ If, as in Example \ref{nnnnnc}, $(T f)(z) = z f(z)$ is the unilateral shift on $H^p$, then $f$ is $T$-inner precisely when 
$$f \perp_{H^p}  z^n f \iff \int_{0}^{2 \pi} |f(e^{i \theta})|^{p} e^{i n \theta} \frac{d \theta}{2 \pi} = 0, \quad n \geq 1.$$ Again, this shows that $f$ has constant modulus on the circle, i.e., inner in the classical sense. 
\end{Example}

\begin{Example}
For the  Bergman spaces $\mathscr{A}^p$ of analytic functions $f$ on $\D$ for which 
$$\int_{\D} |f(z)|^p dA< \infty$$ 
(which is a closed subspace of $L^p(\D, dA)$), we can use \eqref{orutreee} to extend Example  \ref{popoweripeo} and say that $f \in \mathscr{A}^p$ is $T$-inner ($T f = z f$) when 
$$\int_{\D} |f(z)|^p z^{n} d A = 0, \quad n \geq 1.$$
\end{Example}

\begin{Example}
For the space 
$$\ell^{p}_{A} = \left\{f(z) = \sum_{k \geq 0} a_k z^k: \sum_{k \geq 0} |a_k|^{p} < \infty\right\},$$
which turns out to be a well-studied space Banach space of analytic functions on $\D$ (see \cite{MR3714456} for a survey),
the Birkhoff-James orthogonality becomes 
$$f \perp_{\ell^{p}_{A}} g \iff \sum_{k \geq 0} |a_{k}|^{p - 2} \overline{a_k} b_k = 0.$$
The unilateral shift $(T f)(z) = z f$ is an isometry on $\ell^{p}_{A}$ and the notion of $T$-inner was studied in \cite{zeroslp}. The condition for $f \in \ell^{p}_{A}$ to be $T$-inner is 
$$\sum_{k \geq 0} |a_{k}|^{p - 2} \overline{a_k} a_{N + k} = 0, \quad N \geq 1,$$ but this condition can be difficult to work with. One can see functions such as $f(z) = z^{n}$ are inner. When $w \in \D \setminus \{0\}$ an analysis in \cite{CR2} shows that 
$$f(z) =  \frac{1 - z/w}{1 - |w|^{p - 2} \overline{w} z}$$ is inner. Notice how when $p = 2$ this function becomes a constant times the single Blaschke factor
$$\frac{z - w}{1 - \overline{w} z}.$$
\end{Example}


\section{Application: Zero sets for Banach spaces of analytic functions}
In this section we develop the analog of Theorem \ref{v98u24tprefewq23re} for Banach spaces of analytic functions. 
Let $\mathscr{X}$ be a uniformly convex, smooth, complex Banach space of analytic functions on a domain $\Omega$ that satisfies the following conditions.
\begin{equation}\label{X1}
\mbox{Point evaluation of derivatives of any order is bounded};
\end{equation}
\begin{equation}\label{X2}
f \in \mathscr{X} \implies z f(z) \in \mathscr{X};
\end{equation}
\begin{equation}\label{X3}
\bigvee\{z^j: j \geq 0\} =  \mathscr{X};
\end{equation}
\begin{equation}\label{X4}
w \in \Omega, f \in \mathscr{X} \implies \frac{f(z) - f(w)}{z - w} \in \mathscr{X}
\end{equation}
For some positive constants $r$ and $K$, 
\begin{equation}\label{X5} 
 f \perp_{\mathscr{X}} g  \ \implies  \|f+g\|^r \geq \|f\|^r + K\|g\|^r.
\end{equation}

Just as in the Hilbert space case, condition (\ref{X1}) can be deduced from conditions (\ref{X2}) -- (\ref{X4}). Furthermore, conditions \eqref{X1}, \eqref{X2}, and the closed graph theorem show that the operator 
$$S_{\mathscr{X}}: \mathscr{X} \to \mathscr{X}, \quad (S_{\mathscr{X}} f)(z) = z f(z),$$
is a bounded linear operator on $\mathscr{X}$. 

Note that $\mathscr{X}$ is reflexive and enjoys the unique nearest point property in the sense of \eqref{bbbbbx}. Furthermore, each nonzero vector $f \in \mathscr{X}$ has a unique norming functional, $\ell_f$ from which it follows from our general discussion in the previous section that $f \perp_{\mathscr{X}} g$ if and only if
$
   \ell_f(g) = 0.
$
Since evaluation at each point $w \in \Omega$ is continuous, it is given by a functional $k_{w} \in \mathscr{X}^{*}$, i.e., 
$$f(w) = k_{w}(f).$$
Unlike the Hilbert space case discussed earlier, where $k_{\lambda}$ belonged to the Hilbert space (equating Hilbert space with its dual space in the natural way via the Riesz representation theorem), here $k_{w}$ belongs to the dual space $\mathscr{X}^{*}$ which is not necessarily a space of analytic functions (and for which we don't use the notation $k_{w}(z)$ as we did for the Hilbert space case). 

Condition \eqref{X5} is a ``Pythagorean inequality,'' and it was shown in \cite{CR} that all $L^p$ spaces with $p \in (1, \infty)$ satisfy this condition for a range of parameter values $r$ and $K$. Furthermore, the inequality holds in reverse for other values of $r$ and $K$.

Important to the development of the analog of Theorem \ref{v98u24tprefewq23re} for Banach spaces is the following projection lemma, which makes use of the Pythagorean inequality from \eqref{X5}. 

\begin{Lemma}\label{43235t349823refbdsas}
Let $\mathscr{X}$ be a smooth Banach space satisfying (\ref{X5}).  For each $n \in \mathbb{N}$, suppose that $\mathscr{X}_n$ is a subspace of $\mathscr{X}$, such that
     \[
          \mathscr{X}_1 \subseteq \mathscr{X}_2 \subseteq \mathscr{X}_3 \subseteq \cdots.
     \]
     Define $\mathscr{X}_{\infty} = \overline{\bigcup_{n=1}^{\infty} \mathscr{X}_n}$.   If $P_n$ is the metric projection mapping from $\mathscr{X}$ to $\mathscr{X}_n$, for all $n \in \mathbb{N} \cup \{\infty\}$, then for any $\vec{x} \in \mathscr{X}$, $P_n \vec{x}$ converges to $P_{\infty} \vec{x}$ in norm.
\end{Lemma}

\begin{proof}
By hypothesis, $\mathscr{X}$ is uniformly convex (and hence has unique nearest points), and satisfies the Pythagorean inequality
\[
    \| \vec{x} + \vec{y} \|^r \geq \|\vec{x}\|^r + K \|\vec{y}\|^r
\]
whenever $\vec{x} \perp_{\mathscr{X}} \vec{y}$.   Let $\vec{x} \in \mathscr{X}$.  By the definition of metric projection, whenever $m<n$, we have
\begin{align*}
    \| \vec{x} - P_m \vec{x} \| &=  \inf\{ \|\vec{x}-\vec{z}\| : \vec{z} \in \mathscr{X}_m\} \\
       &\geq \inf\{ \|\vec{x}-\vec{z}\| : z \in \mathscr{X}_n\}\\
       &= \|\vec{x} - P_n \vec{x}\|\\
       &\geq \|\vec{x} - P_{\infty} \vec{x}\|.
\end{align*}
Thus, as a sequence indexed by $n$, $\|\vec{x} - P_n \vec{x}\|$ is monotone nonincreasing, and bounded below.  Accordingly, it converges.  

Next, for $m<n$, the vector $P_n \vec{x} - P_m \vec{x}$ lies in $\mathscr{X}_n$ (the larger space), and hence the co-projection $\vec{x} - P_n \vec{x}$ is Birkhoff-James orthogonal to it.
Consequently,
the Pythagorean inequality says that
\[
      \| \vec{x} - P_m \vec{x} \|^r \geq \|\vec{x} - P_n \vec{x} \|^r + K\|P_n \vec{x} - P_m \vec{x}\|^r.
\]
Since the (positive) difference $\| \vec{x} - P_m \vec{x} \|^r  - \|\vec{x} - P_n \vec{x} \|^r$ can be made arbitrarily small by choosing $m$ sufficiently large, it follows that $\{P_m \vec{x}\}_{m\geq 1}$ is a Cauchy sequence in norm, and converges to some vector $\vec{z}$.  It is clear that $\vec{z} \in \mathscr{X}_{\infty}$, and hence 
\[
      \| \vec{x}-\vec{z} \|  \geq \| \vec{x} - P_{\infty}\vec{x}\|.
\]

Next, let $\epsilon >0$.  There exists an $N$ such that 
\[
         \|\vec{x} -  \vec{y} \|   \leq  \| \vec{x} - P_{\infty}\vec{x}\| + \epsilon
\]
for some $\vec{y} \in \mathscr{X}_N$.  But then
\[
       \|\vec{x}-\vec{z} \| \leq \|\vec{x} - P_n \vec{x}\| \leq \|\vec{x}-\vec{y}\| \leq  \| \vec{x} - P_{\infty}\vec{x}\| + \epsilon.
\]
Since this is true for arbitrary $\epsilon$, we conclude that \[
      \| \vec{x}-\vec{z} \|  \leq \| \vec{x} - P_{\infty}\vec{x}\|.
\]  Equality holds in these norms, so finally uniqueness of nearest points forces $\vec{z} = P_{\infty}\vec{x}$.
\end{proof}

With the above set up we are now ready to develop a version of Theorem \ref{v98u24tprefewq23re} for Banach spaces satisfying the conditions \eqref{X1} - \eqref{X5}.
Fix an infinite sequence $W = (w_1, w_2, w_3,\ldots) \subset \Omega \setminus \{0\}$, and for each $n\geq 1$, define
$$
     f_n(z) =  \Big( 1-\frac{z}{w_1}  \Big)\cdots\Big( 1-\frac{z}{w_n}  \Big)$$
and, by Proposition \ref{v04rkjbfvj}, the $S_{\mathscr{X}}$-inner function 
$$     J_n =  f_n - \widehat{f}_n,
$$
where $\widehat{f}$ stands for the metric projection of $f$ onto $[zf]$.  Note that $\widehat{f}$ exists and is unique, by uniform convexity. (When $\mathscr{X}$ is a Hilbert space, the metric projection coincides with the orthogonal projection.)

Let $k_j \in \mathscr{X}^*$ denote the evaluation functional at $w_j$, $j \geq 1$ and $k_0$ denote the evaluation functional at the origin.  The analogous argument used to prove Lemma \ref{sldkfjsdfeq} shows that 
\[
       \{ f \in \mathscr{X}:  f(0) = f(w_j) = 0, 1 \leq j \leq n\}  = [z f_n].
\]

Next, suppose that $\lambda = \ell_{J_n} \in \mathscr{X}^*$ is the norming functional for $J_n$.  From
\[
      J_n \perp _{\mathscr{X}}  z^j f_n, \quad 1\leq j \leq n
\]
and \eqref{xxxkldfkdlk} we see that 
\[
      \lambda(z^j f_n) = 0,\quad 1\leq j \leq n.
\]
That is, $\lambda \in [z f_n]^{\perp} = \bigvee\{k_j : 0 \leq j \leq n \}$.  We may therefore express $\lambda$ as
\[
      \lambda = c_0 k_0 + c_1 k_1 + \cdots + c_n k_n
\]
for some complex coefficients $c_0$, $c_1$,\ldots, $c_n$.   By definition of norming functional this says that
\begin{align*}
    \|J_n\| &=   \lambda(J_n) \\
              &=  c_0 k_0(J_n) + c_1 k_1(J_n) + \cdots+ c_n k_n(J_n)\\
              &= c_0 \cdot 1 + 0 + \cdots + 0\\
              &= c_0,
\end{align*}
since $k_0(J_n) = J_n(0) = f_n(0) =1$.

Finally, the condition
\[
      k_j(J_n) = 0, \ 1 \leq j \leq n
\]
can be interpreted as saying that
\[
      \lambda \perp_{\mathscr{X}^*}  k_j,\ 1 \leq j \leq n.
\]
That is, $\lambda$ solves the infimum problem
\[
      \inf \| c_0 k_0 + c'_1 k_1 + \cdots + c'_n k_n \|
\]
where $c_0 = \|J_n\|$ is fixed, and $c'_1,\ldots,c'_n$ are varied.

By renaming the constants, we have shown that 
\[
      1 = \|\lambda\| = \|J_n\| \inf \| k_0 + b_1 k_1 + \cdots + b_n k_n \|,
\]
or
\[
         \|J_n\|  = \Big[  \inf \| k_0 + b_1 k_1 + \cdots + b_n k_n \|  \Big]^{-1}.
\]
As $n$ tends to infinity, the infimum is over a larger set, and thus decreases monotonically, while $\|J_n\|$ must therefore be nondecreasing monotonically.  

Suppose $W$ is the zero set of some nontrivial function $f \in \mathscr{X}$.  By dividing by $z$ a suitable number of times, we can assume that $f(0) \neq 0$.  Then 
\begin{align*}
\| k_0 + b_1 k_1 + \cdots + b_n k_n \| &  \geq | \langle k_0 + b_1 k_1 + \cdots + b_n k_n, f\rangle | /\|f\|\\
& = |f(0)|/\|f\|
\end{align*}
is bounded from zero, and consequently $\|J_n\|$ is bounded above.

Conversely, if $W$ fails to be the zero set of some nontrivial function of $\mathscr{X}$, then by the Lemma \ref{43235t349823refbdsas}, there exists an element $\Lambda$ of $\mathscr{X}^*$ such that the following infimum is attained:
$$
      \|\Lambda\| =  \inf\{ \| k_0 + b_1 k_1 + \cdots + b_n k_n \|:  b_1,\ldots, b_n \in \mathbb{C}, n\geq 1\}.
      $$ Indeed, Lemma \ref{43235t349823refbdsas} tells us that $\Lambda$ is the $k_0$ minus its metric projection onto 
      $$\bigvee\{k_1, k_2, k_3,\ldots\}.$$

Let $\Phi \in \mathscr{X}$ be the norming functional of $\Lambda$.  Then the infimum condition assures that
\[
   \Lambda \perp_{\mathscr{X}^*}  k_j, \ 1\leq j;
\]
that is,
\[
       k_j(\Phi) = 0
\]
for all $j \geq 1$.  This shows that $W$ is a zero set for $\Phi$.  The only way this can happen is if $\Phi$ is identically zero, which implies that 
\[
    \lim_{n\rightarrow\infty} \inf \| k_0 + b_1 k_1 + \cdots + b_n k_n \|  = 0.
\]
We memorialize these findings as follows, obtaining an extension of Theorem \ref{v98u24tprefewq23re} to certain Banach spaces of analytic functions.
\begin{Theorem}\label{v98u24tprefewq23re2}
Let $\mathscr{X}$ be a uniformly convex, smooth, complex Banach space of analytic functions on a domain $\Omega$ satisfying conditions (\ref{X1}) -- (\ref{X5}).
Let $(w_j)_{j \geq 1} \subset \Omega \setminus \{0\}$ and 
$$f_n = \prod_{j = 1}^{n} \Big(1 - \frac{z}{w_j}\Big), \quad J_{n} = f_{n} - P_{[z f_n]} f_n.$$
Then 
\begin{enumerate}
\item Each $J_n$ is an $S_{\mathscr{X}}$-inner function;
\item the sequence $\|J_n\|$ is a non-decreasing sequence; 
\item $(w_j)_{j \geq 1}$ is a zero sequence for $\mathscr{X}$ if and only if 
$$\sup\{\|J_n\|: n \geq 1\} < \infty.$$
\end{enumerate}
\end{Theorem}

Spaces for which this applies (i.e., they satisfy the conditions of the abstract Banach space along with the Pythagorean inequality) include the $L^p$ Bergman spaces and $\ell^{p}_{A}$ spaces. A proof specifically tailored for $\ell^{p}_{A}$ was developed in \cite{zeroslp}.


\bibliographystyle{plain}
\bibliography{references6}
\end{document}